\documentclass[12pt, onesided]{amsart}

\usepackage{amsmath, amsthm, amssymb, graphicx, amsfonts, hyperref,cleveref}
\usepackage{amsrefs}
\usepackage{color}
\usepackage{xypic, pinlabel, multicol, longtable}
\usepackage[all]{xy}
\usepackage{tikz}
\usepackage{enumerate}
\usepackage{booktabs}

\newtheorem{lemma}{Lemma}[section]
\newtheorem{thm}[lemma]{Theorem}
\newtheorem{prop}[lemma]{Proposition}
 
\newtheorem{conj}[lemma]{Conjecture}
\newtheorem{mainthm}{Theorem}

\newtheorem{maincor}[mainthm]{Corollary}

\usepackage[margin=3cm, marginpar=2cm]{geometry}

\theoremstyle{definition}
\newtheorem{defn}[lemma]{Definition}
\newtheorem{rmk}[lemma]{Remark}
\newtheorem{quest}[lemma]{Question}

\newtheorem{example}[lemma]{Example}

\title{The Farey tree and embeddings of lens spaces and rational balls in $\CP$}

\author{Marco Golla}
\address{CNRS and Nantes Universit\'e, Nantes, France}
\email{marco.golla@univ-nantes.fr}

\author[Brendan Owens]{Brendan Owens} 
\address{School of Mathematics and Statistics, University of Glasgow, University Place, Glasgow G12 8QQ, United Kingdom}
\email{brendan.owens@glasgow.ac.uk}

\date{}

\newsavebox\ltmcbox

\newcommand{\CP}{\mathbb{CP}\vphantom{\mathbb{CP}}^2}
\newcommand{\CPb}{\overline{\mathbb{CP}}\vphantom{\CP}^2}
\newcommand{\Bpq}{B_{p,q}}
\newcommand{\Z}{\mathbb{Z}}

\newcommand{\R}{\mathbb{R}}

\newcommand{\PP}{\mathbb{P}}
\newcommand{\C}{\mathbb{C}}

\newcommand{\de}{\partial}
\newcommand\sd{\mkern1.5mu{:}\mkern1.5mu}

\newcommand{\triple}[9]{$\left(\left(\frac{#1}{#2},#3\right),\left(\frac{#4}{#5},#6\right),\left(\frac{#7}{#8},#9\right)\right)$}

\begin{document}

\maketitle

\begin{abstract}
Motivated by a conjecture of Koll\'ar, we study embeddings of multiple rational homology balls in $\CP$.
To each node of the Farey tree, we associate such an embedding of three rational homology balls with lens space boundary, extending earlier work of the second author and of Lisca and Parma, using a recursive Kirby calculus argument.
We also give further explicit constructions of embeddings of triples of rational homology balls into homotopy $\CP$s.
\end{abstract}

\section{Introduction}\label{s:intro}
A basic, yet fundamental, question in 4-manifold topology is the following.
Given a smooth 4-manifold $X$ and a class $\mathcal{Y}$ of 3-manifolds, which $Y \in \mathcal{Y}$ smoothly embed in $X$?
In this paper, we focus on the case when $\mathcal{Y}$ is the class of lens spaces, which are the simplest 3-manifolds after the sphere, and on the simplest possible 4-manifold in which they can embed\footnote{It is known since Hantzsche's work~\cite{Hantzsche} that lens spaces cannot embed in $S^4$.}: the complex projective plane $\CP$.

Whenever a lens space $L$ embeds in $\CP$, it splits it into two components, one of which is a \emph{rational homology ball}, a 4-manifold with the rational homology of a point.
Therefore, asking for the embedding of $L$ into $\CP$ is equivalent to asking for the embedding into $\CP$ of a rational ball whose boundary is $L$.
Note that lens spaces that bound rational homology balls have been classified by Lisca~\cite{Lisca-ribbon}.

Early results of the existence of this type of embedding come from algebraic geometry~\cite{Manetti, hp} and symplectic topology~\cite{es}, as well as from~\cite{nonsymp,LPStein,HD1, HD2,EMPR}; see below for a more detailed historical account.
These results indicate that it is interesting to look at embeddings of \emph{triples} of rational homology balls in $\CP$, a question which is related  (but not equivalent) to asking about disjoint embeddings of lens spaces.

Here we push the ideas of~\cite{HD2} and of~\cite{EMPR} further, and we obtain many more embeddings of \emph{triples} of rational homology balls with lens space boundary in $\CPb$;
reversing all orientations gives embeddings in $\CP$.
Our triples give previously unknown examples, except possibly when $\min\{p_i\} \le 2$.

\begin{mainthm}\label{thm:ANN}
If $p_1,p_2,q_1,q_2$ are nonnegative integers with $q_1, q_2$ even, $p_i\ge q_i$, and $p_1q_2-p_2q_1=\pm2$, then the disjoint union
\[
B_{p_1,q_1}\sqcup B_{p_2,q_2}\sqcup B_{p_1+p_2,q_1+q_2}
\]
admits a smooth orientation-preserving embedding in $\CPb$.
\end{mainthm}

When $p$ and $q$ are coprime or have greatest common divisor 2, $B_{p,q}$ is a rational homology ball constructed with only one 1-handle and one 2-handle, whose boundary is $L(p^2,pq-1)$.
Kirby diagrams for these rational balls are given in Figures~\ref{fig:Bpq} and~\ref{fig:Apq}.

An immediate corollary of Theorem~\ref{thm:ANN} above is the following.

\begin{maincor}\label{cor:Bpq}
Let $p\ge q$ be nonnegative integers.  If $\gcd(p,q)=2$ or if $p$ is odd and $\gcd(p,q)=1$, then the rational ball $B_{p,q}$ admits a smooth orientation-preserving embedding in $\CPb$.
\end{maincor}

In terms of embeddings of lens spaces, the following corollary shows that, among lens spaces of the form $L(p^2,pq-1)$ that bound a rational homology ball,
roughly three-quarters  embed in $\CP$.

\begin{maincor}\label{cor:Lpq}
Let $p\ge q$ be nonnegative integers.  If $\gcd(p,q)=2$ or if $p$ is odd and $\gcd(p,q)=1$, then the lens space $L(p^2,pq-1)$ admits a smooth embedding in $\CP$.
\end{maincor}

The inspiration for the main theorem comes from the data extracted from the output of two more general constructions, that we describe in detail in Section~\ref{s:homotopyCP2s}, which give a large collection of embeddings of triples of rational homology balls in  homotopy $\CP$'s.
We expect that all these homotopy $\CP$'s are in fact standard (see \Cref{q:exoticCP2s} below).

In a companion paper~\cite{GOobst},  we provide obstructions to the existence of embeddings of such triples, and work in the direction of  a conjecture of Koll\'ar \cite{Kollar}; see also~\cite{JPP} for related recent work. 
It is worth noting that the embeddings of triples of Theorem~\ref{thm:ANN} account for a very large fraction of the unobstructed triples.

The proof of Theorem~\ref{thm:ANN} is inductive, and relies on a tricky, but rather direct, Kirby calculus argument.
It uses structural induction on a simple variant of the Farey tree, which we will define in Section~\ref{s:2Fareytree}. 
The Kirby calculus argument is related to Lisca and Parma's horizontal decompositions, and is informed by the Kirby diagrams of~\cite{EMPR}.
The main idea is to reduce all computations to handle slides in the neighbourhood of a 2-torus in the boundary of the 1-handlebody $S^1\times D^3$.

The handlesliding used to prove Theorem~\ref{thm:ANN} suggests another point of view on the results of Lisca and Parma in \cite{HD2}.  In \Cref{thm:trees} we use this point of view to derive new recursive descriptions of the triples of rational balls embedded in $\CP$ obtained in that paper.

\subsection*{History and motivation}
As mentioned above, the motivation for studying embeddings of lens spaces in $\CP$ comes from algebraic geometry.
Inspired by consequences of the orbifold Bogomolov--Miyaoka--Yau inequality, Koll\'ar asks in~\cite{Kollar} whether every compact, orientable, simply-connected 4-manifold $M$ with $H_1(M) = 0$, $H_2(M)\cong\Z$, and with boundary a union of spherical 3-manifolds $Y_1, \dots, Y_n$, satisfies:
\[
\sum_{i=1}^n \left(1-\frac1{|\pi_1(Y_i)|}\right) \le 3.
\]
Lens spaces are finite cyclic quotients of $S^3$, so they are spherical 3-manifolds, and indeed the complements of the rational homology balls that we embed in $\CP$ are manifolds satisfying the assumption in Koll\'ar's question.

A prominent family of examples of such manifolds originates from weighted projective planes $\PP(a \sd b \sd c)$ by removing small open neighbourhoods of their (potentially) singular points $(0\sd 0\sd 1)$, $(0\sd 1\sd 0)$, and $(1\sd 0\sd 0)$.
When $(a,b,c)$ satisfies $a^2 + b^2 + c^2 = 3abc$, the weighted projective plane $\PP(a^2 \sd b^2 \sd c^2)$ admits a deformation to $\CP$.
Such triples are called \emph{Markov triples}, and the deformation gives rise to an embedding of three rational homology balls, each with boundary a lens space, into $\CP$.
Whenever an embedding of this type exists, the orders of the fundamental groups are coprime perfect squares, and the inequality conjectured by Koll\'ar implies that $M$ has at most three boundary components.

In the realm of complex algebraic geometry, this is in fact the only instance of such embeddings, as was shown by Hacking and Prokhorov~\cite{hp}.
The corresponding result in the symplectic category is also known to hold, by work of Evans and Smith~\cite{es}.
Explicit Kirby diagrams for the associated 4-manifolds with boundary were exhibited by Etnyre, Min, Piccirillo, and Roy~\cite{EMPR}.
Topologically, however, the situation is quite different: earlier work of the second author~\cite{nonsymp},
 and of Lisca and Parma~\cite{LPStein,HD1, HD2}, shows that there are infinitely many embeddings of rational balls in $\CP$ that do not arise symplectically.

We would like to stress that, among the rational homology balls $\pm B_{p,q}$, the only ones that carry some `good' complex or symplectic structure are the those for which $\gcd(p,q) = 1$, with the positive sign.
By `good' here we mean that either they are diffeomorphic to a Milnor fibre in the complex setup \cite{wahl}, or that they have a symplectic structure with respect to which the boundary is convex in the symplectic setup \cite{LiscaSympFill}.
In particular, since in Theorem~\ref{thm:ANN} we have embeddings of $B_{p,q}$ in $\CPb$, none of the nontrival rational homology balls in any of the triples, with the exception of $B_{2,0}$, bear any complex or symplectic significance.

\subsection*{Notation and conventions}
Unless otherwise stated, homology is taken with integer coefficients. The lens space $L(p,q)$ is the quotient of $S^3 \subset \C^2$ by the action of the group generated by  $\begin{pmatrix}\zeta & 0 \\ 0 & \zeta^q\end{pmatrix}$, where $\zeta$ is a primitive $p^{\rm th}$ root of $1$, which is also the link of the cyclic quotient singularity of type $\frac1p(1,q)$. With this description, $L(p,q)$ is obtained by doing $-p/q$-surgery on the unknot in $S^3$.

\subsection*{Organisation of the paper}
In Section~\ref{s:balls} we give different descriptions of the balls $B_{p,q}$.
In Section~\ref{s:2Fareytree} we introduce the 2-Farey tree and state some of its properties, which we use in Section~\ref{s:ANN} to prove Theorem~\ref{thm:ANN}.
In Section~\ref{s:compare} we put the results from Theorem~\ref{thm:ANN} in a common framework with those of Lisca and Parma~\cite{HD2}, and also give new recursive descriptions of the examples obtained in \cite{HD2}.
In Section~\ref{s:homotopyCP2s} we describe two constructions of embeddings of triples of rational balls, coming from plumbings, and one construction of embeddings of rational balls and pairs of rational balls, coming from knots with lens space surgeries.
Finally, in Section~\ref{s:data} we share and discuss some data.

\subsection*{Acknowledgements}
We would like to warmly thank Yank{\i} Lekili for a question which initiated this collaboration.
We also thank Nikolas Adaloglou, Giulia Carfora, Johannes Hauber, Woohyeok Jo, Paolo Lisca, Marco Marengon, Jongil Park, and Kyungbae Park for stimulating conversations.
We would like to thank Andrea Parma for providing helpful data and clarifications about his papers.
The first author thanks the University of Glasgow for their hospitality.
The second thanks Oxford University and Nantes University for their hospitality.
Part of this work has been carried out at several conferences and workshops, including \emph{Surfaces in 4-manifolds} in le Croisic, \emph{4-vari\'et\'es: \`a travers les dimensions}
at CIRM in Marseilles, \emph{The low-dimensional workshop} during the \emph{Singularities and low-dimensional topology} research semester at the Erd\H{o}s Center in Budapest and \emph{Combinatorial and gauge-theoretical methods in low-dimensional topology and geometry} at the Centro De Giorgi in Pisa.


\section{The balls $\Bpq$}\label{s:balls}
We use Hirzebruch--Jung continued fractions, with the following notation:
$$[a_1,a_2,\ldots,a_k]
:=a_1-\frac{1}{a_2-\raisebox{-3mm}{$\ddots$
\raisebox{-2mm}{${-\frac{1}{\displaystyle{a_k}}}$}}}.$$

We will also briefly make use of Euclidean continued fractions, with 
$$[n_1,n_2,\ldots,n_m]^+
:=n_1+\frac{1}{n_2+\raisebox{-3mm}{$\ddots$
\raisebox{-2mm}{${+\frac{1}{\displaystyle{n_m}}}$}}}.$$

Let $p>q>0$ be integers with $\gcd(p,q)\in\{1,2\}$.  We suppose that
$$\frac{p}{q}=[a_1,\dots,a_k]$$
and
$$\frac{p}{p-q}=[b_1,\dots,b_\ell].$$
with $a_i,b_j\ge2$ for all $i,j$.

If $(p,q)=(1,0)$ then we set $\dfrac{p}{q}=[\,]$ (the empty continued fraction) and $\dfrac{p}{p-q}=[1]$.

If $(p,q)=(2,2)$ then  $\dfrac{p}{q}=[1]$ and $\dfrac{p}{p-q}=[\,]$.

Two Kirby diagrams of rational homology balls bounded by lens spaces  are shown in each of Figures \ref{fig:Bpq} and \ref{fig:Apq}.
We refer to the framed blue attaching circle, in the second diagram of each of the figures,  as $K_{p,q}$.  When $p$ and $q$ are coprime, $K_{p,q}$ is a (negative) torus knot,
which we may take to lie on a torus at distance $r>0$ from the dotted circle in the diagram.  When $\gcd(p,q)=2$, then $K_{p,q}$ is a $(2,-1)$ cable of a torus knot, which lies in an $\varepsilon$-neighbourhood, due to the cabling crossing, of a torus at distance $r>0$ from the dotted circle.  
We note that in each case, $K_{p,q}$ is unoriented, and $K_{p,q}=K_{-p,-q}$.

In each case we refer to the rational ball shown as $B_{p,q}$.  This is partly justified by the following lemma.  The case of $p,q$ coprime was previously established by \cite{lm} using symplectic methods.  The lemma may also be proved using \cite[Lemma 2.16]{ER}.

\begin{figure}[htbp]
\centering
\includegraphics[width=14cm]{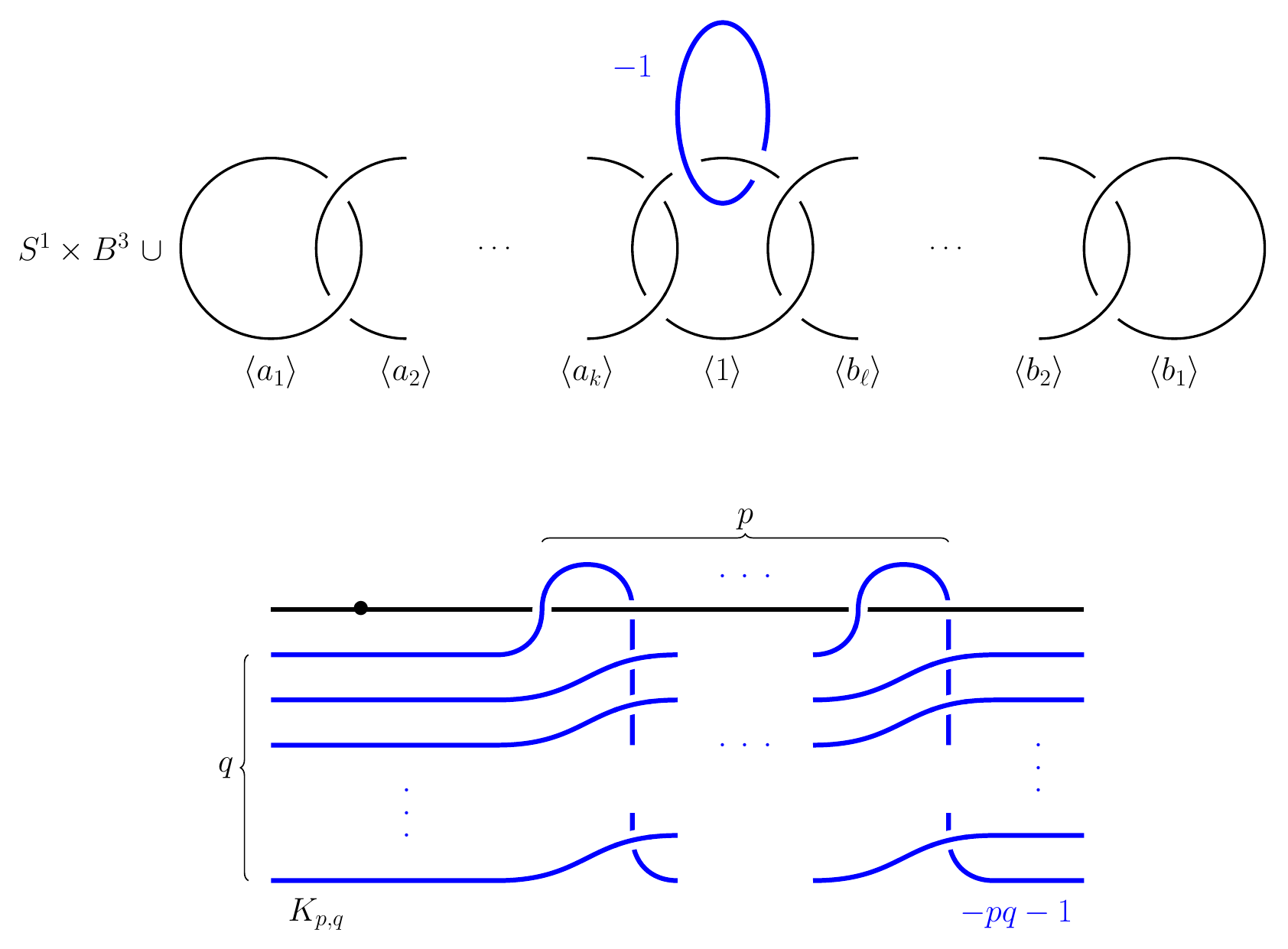}
\caption{{\bf Two Kirby diagrams of the rational homology ball $B_{p,q}$ when $\gcd(p,q)=1$.} The second diagram should be closed like a braid closure, so that there is one dotted circle in black and one 2-handle attaching circle, which is a copy of the torus knot $T_{-p,q}$, in blue.  The framing on the 2-handle is $-1$ relative to the torus framing.}
\label{fig:Bpq}
\end{figure}

\begin{figure}[htbp]
\centering
\includegraphics[width=14cm]{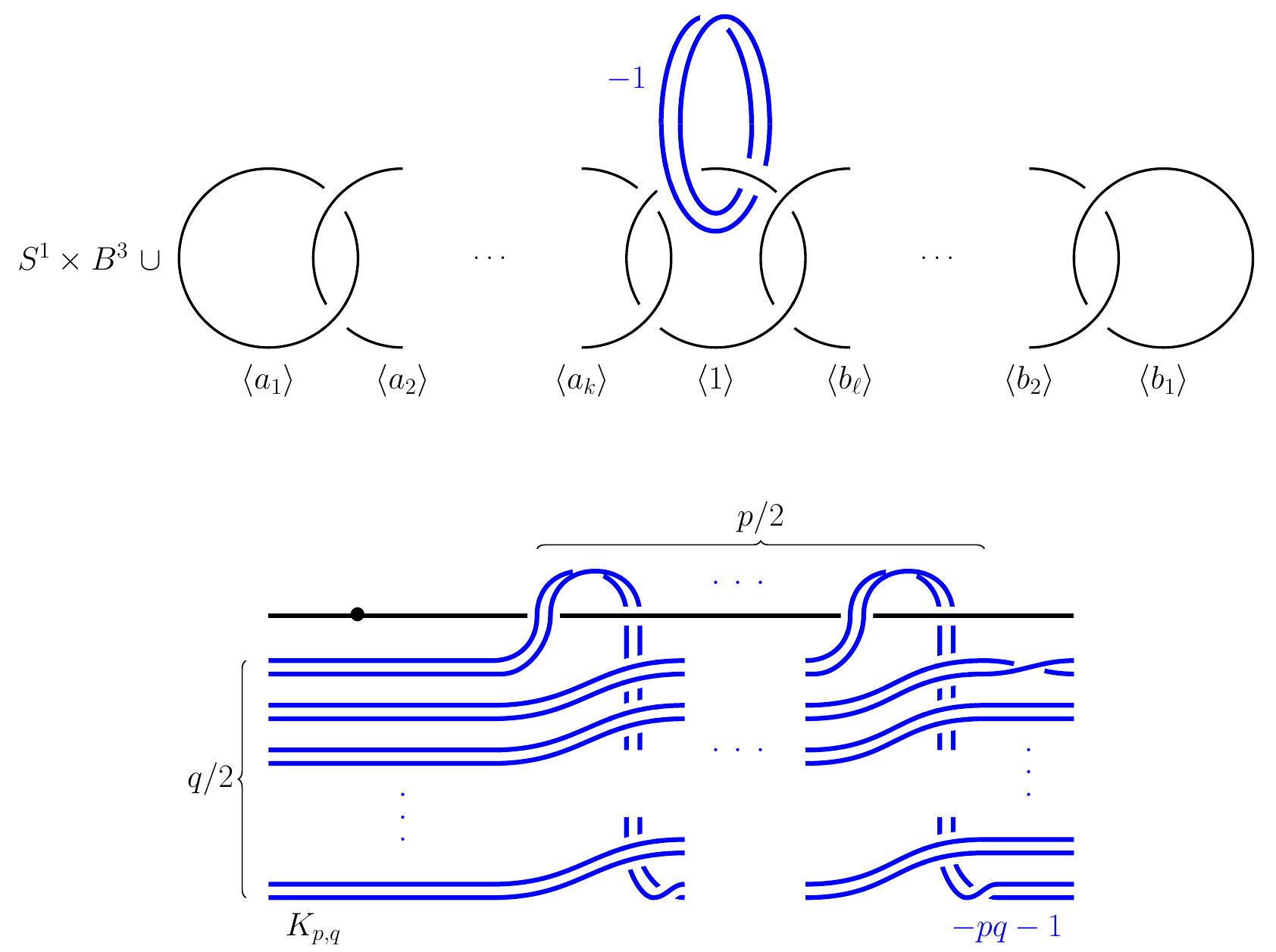}
\caption{{\bf Two Kirby diagrams of the rational homology ball $B_{p,q}$ when $\gcd(p,q)=2$.} The blue attaching circle for the 2-handle in the second diagram is the $(2,-1)$ cable, relative to the torus framing, of the torus knot $T_{-p/2,q/2}$.}
\label{fig:Apq}
\end{figure}

\begin{lemma}\label{lem:Bpq}
If $p$ and $q$ are coprime, then the two Kirby diagrams shown in \Cref{fig:Bpq}  represent the same diffeomorphism type of rational homology ball, with boundary $L(p^2,pq-1)$.

If $\gcd(p,q)=2$, then the diagrams shown in \Cref{fig:Apq}  represent the same diffeomorphism type of rational homology ball, with boundary $L(p^2,pq-1)$.
\end{lemma}

\begin{proof}
We modify the second diagram in \Cref{fig:Bpq} by a sequence of blow-ups to obtain the first.  We begin by replacing the dotted circle by a bracketed 0-framed circle, yielding a diagram representing the attachment of a single 2-handle to the boundary of $S^1\times B^3$, similar to the first.  The blue attaching circle of the 2-handle is $T_{-p,q}$; it runs $p$ times around the longitude and $-q$ times around the meridian of a torus which is the boundary of a neighbourhood of the black circle.  We note that the same torus bounds a neighbourhood of a meridian of the black circle (not shown in the figure) and that the longitude and meridian form a positive Hopf link, so that either is a right-handed meridian of the other.  This yields the standard identification $T_{-p,q}\simeq T_{-q,p}$.  Thus a positive blow-up along a meridian external to this torus modifies the knot type of the blue curve from $T_{-p,q}$ to $T_{-p+q,q}$, and increases the framing by $q^2$.
Put another way, it results in
$$K_{p,q}\leadsto K_{p-q,q}.$$
It also increases the framing of an internal longitude by 1.
Similarly a positive blow-up along an internal longitude results in
$$K_{p,q}\leadsto K_{p,q-p},$$
whilst also increasing the framing of an external meridian by 1.

Suppose that the Euclidean continued fraction of $p/q$ is
$$\frac{p}{q}=[n_1,\dots,n_m]^+.$$
This corresponds to the following Euclidean algorithm for $p/q$:
\begin{align}\label{eq:Euc}
p&=n_1\times q+r_1\notag\\
q&=n_2\times r_1+r_2\\
&\,\,\,\vdots\notag\\
r_{m-2}&=n_m\times 1+0\notag.
\end{align}

This Euclidean algorithm will guide a series of bracketed $+1$-blow-ups which will convert the blue curve to a $-1$-framed unknot.  
We refer the reader to \Cref{fig:B2310} for an example.
We begin by blowing up along an external meridian.  As noted above, this converts the framed blue curve to $K_{p-q,q}$.  We do this a total of $n_1$ times, each time sliding the previous $+1$-framed unknot over the new one.  This results in a blue $K_{r_1,q}$.  The framing sequence on the chain of linked bracketed circles is modified as follows:
$$0 \leadsto 1,1 \leadsto 2,1,2 \leadsto \dots \leadsto n_1,1,2^{[n_1-1]},$$
where $2^{[n]}$ is short hand for a subsequence $2,2,\dots,2$ of length $n$.  The blue curve has linking number $\pm r_1$ with the $n_1$-framed bracketed curve and $\pm q$ with the $+1$-framed bracketed curve, which is an external meridian,  and is split disjoint from the others.

\begin{figure}[htbp]
\centering
\includegraphics[width=14cm]{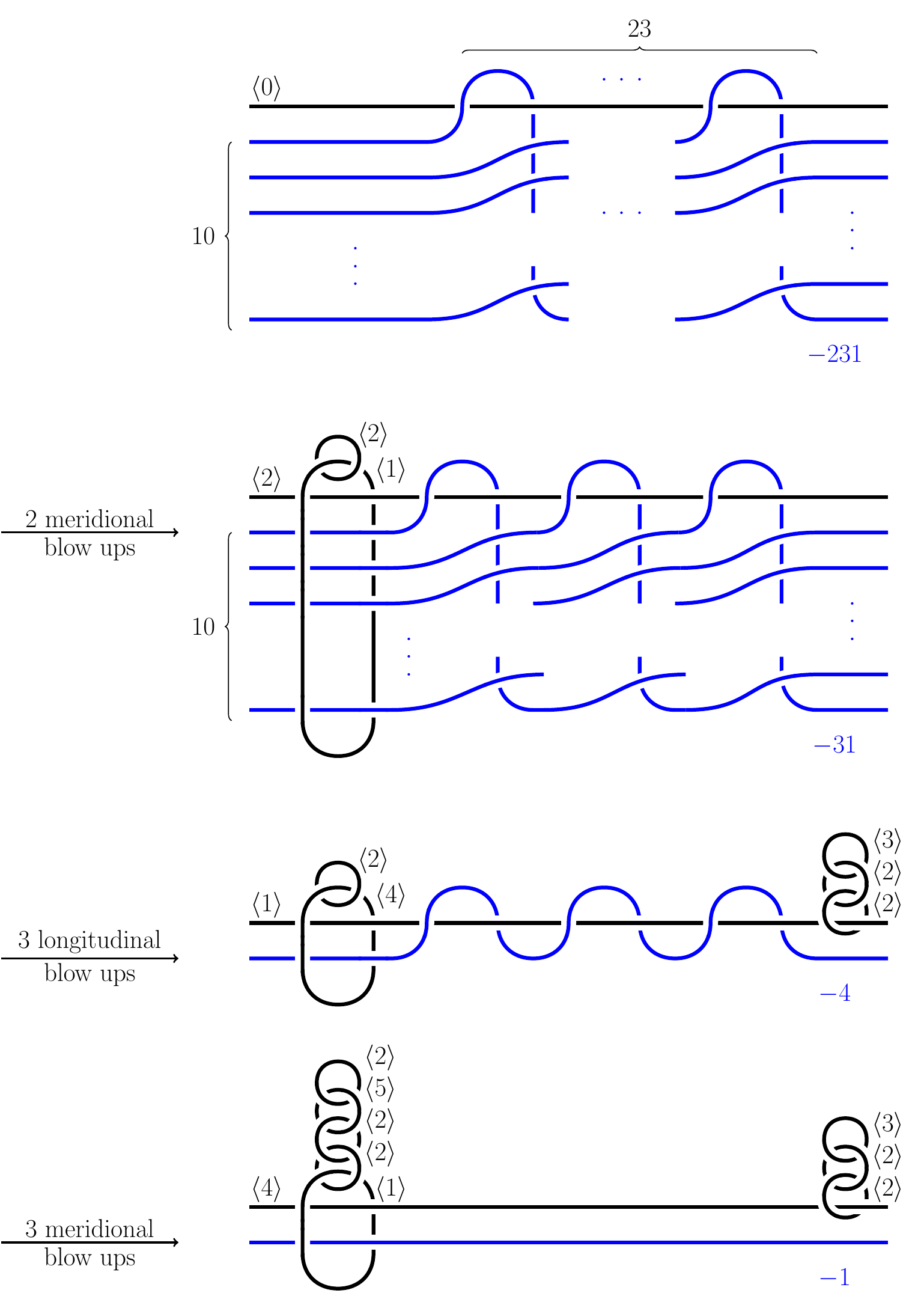}
\caption{{\bf Kirby diagrams of the rational homology ball $B_{23,10}$.}  In each case we take a union with $S^1\times B^3$.}
\label{fig:B2310}
\end{figure}

We then proceed to blow up along an internal longitude $n_2$ times, converting the blue curve to $K_{r_1,r_2}$.  The framing sequence on the bracketed circles gets modified as follows:
$$n_1,1,2^{[n_1-1]} \leadsto n_1+1,1,2,2^{[n_1-1]}  \leadsto  \dots \leadsto n_1+1,2^{[n_2-1]},1,n_2+1,2^{[n_1-1]}.$$
The blue curve has linking number $\pm r_2$ with the $n_2+1$-framed bracketed curve and $\pm r_1$ with the $+1$-framed bracketed curve, which is an internal longitude, and is split disjoint from the others.

We then blow up along an external meridian $n_3$ times, and so on.

After the final $n_m$ blow ups, the bracketed circles have framing sequence
$$
\begin{cases}
n_1+1,2^{[n_2-1]},n_3+2,2^{[n_4-1]}\dots,2^{[n_m-1]},1,n_m+1,2^{[n_{m-1}-1]},\dots,2^{[n_1-1]}& m \text{ even}\\
n_1+1,2^{[n_2-1]},n_3+2,2^{[n_4-1]}\dots,n_m+1,1,2^{[n_m-1]},n_{m-1}+2,\dots,2^{[n_1-1]}& m \text{ odd.}
\end{cases}
$$
 
 Moreover the blue attaching circle of the 2-handle is now a $-1$-framed meridian of the $+1$-framed bracketed circle.  The fact that the framing is $-1$ follows from the identity
 $$pq=n_1q^2+n_2r_1^2+\dots+n_{m-1}r_{m-2}^2+n_m,$$
 which in turn follows easily from \eqref{eq:Euc}.
 
 That this agrees with the first diagram in \Cref{fig:Bpq} follows from the formulas to convert from the Euclidean continued fraction of $p/q$ to the Hirzebruch--Jung continued fractions of $p/q$ and $p/(p-q)$ as in
\cite[Propositions 2.3, 2.7]{PPP}.

The proof in the case $\gcd(p,q)=2$ is entirely similar, again with a sequence of blow-ups guided by the Euclidean algorithm for $p/q=(p/2)/(q/2)$.
In order to see that the $(2,-1)$ cabling with respect to the torus framing is preserved, it helps to note that the effect of blowing up is the same as applying a Dehn twist times the identity to the tubular neighbourhood of the torus which contains the blue curve.

An example is shown in \Cref{fig:A4620}.

To see  that the boundary is $L(p^2,pq-1)$, we  interpret the first diagram in each case as a surgery diagram for a 3-manifold, removing the brackets and the $S^1\times B^3$.  Blow down the blue unknot in each case, which adds 1 or 4 to the framing of the linked bracketed circle.  Then see \cite[Lemma 3.1]{balls} for the case $\gcd(p,q)=1$.  The other case follows in entirely the same way, allowing $(p/2,q/2)$ to take values in the inverse Stern--Brocot tree; the base case with $(p,q)=(4,2)$ involves the continued fraction
$$\frac{16}{9}=[2,5,2],$$
and the rest of the proof is the same.
\end{proof}

\begin{figure}[htbp]
\centering
\includegraphics[width=14cm]{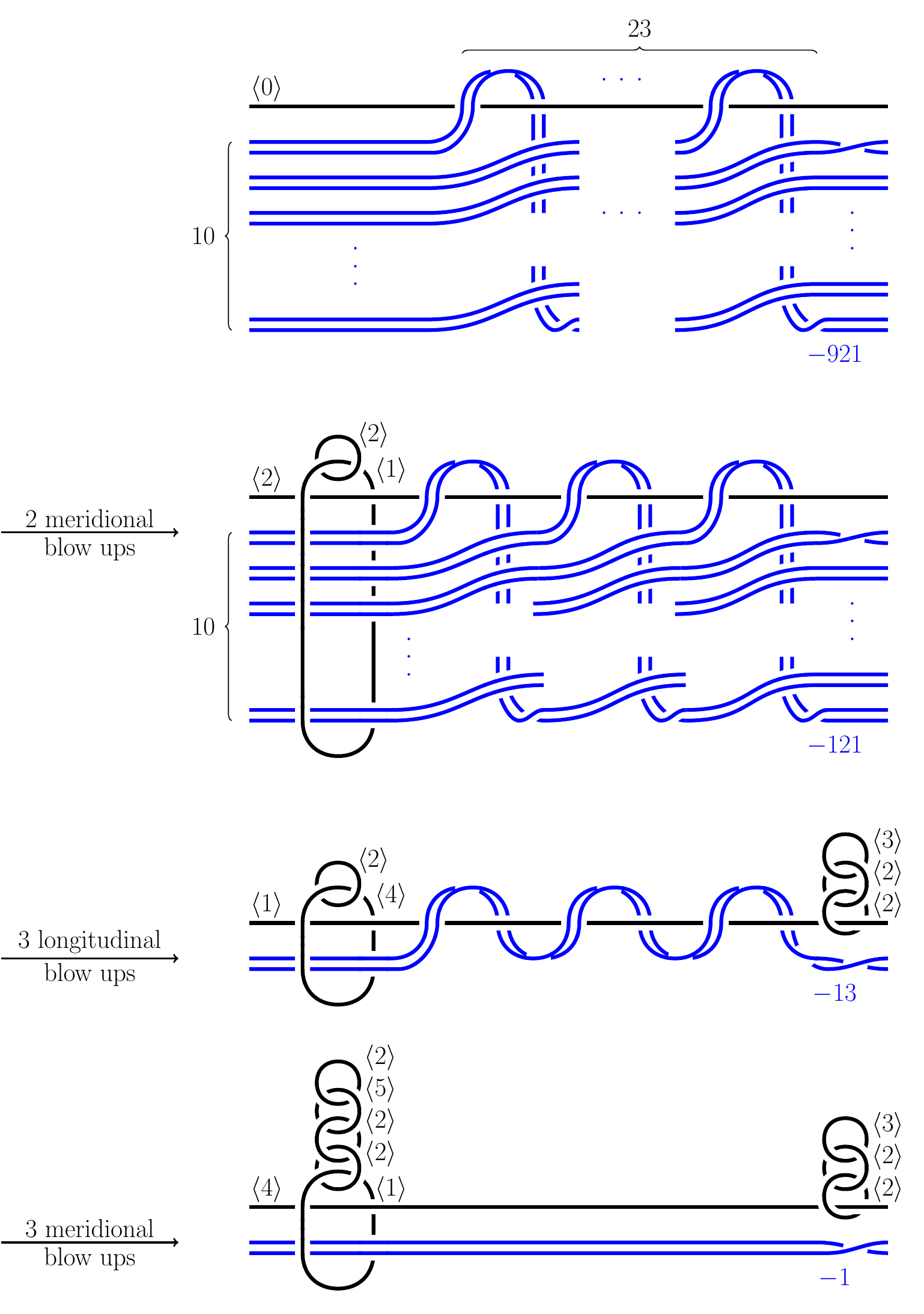}
\caption{{\bf Kirby diagrams of the rational homology ball $B_{46,20}$.}  In each case we take a union with $S^1\times B^3$.}
\label{fig:A4620}
\end{figure}

We finish this section by recording another lemma which is no doubt familiar to experts.  We generalise the definition of $B_{p,q}$ to allow $p$ and $q$ to be arbitrary integers, not both zero, with greatest common divisor $1$ or $2$: take the second diagram of either \Cref{fig:Bpq} or \ref{fig:Apq} as the definition in this case.

\begin{lemma}\label{lem:addp}
Given $p\ge q\ge0$ with $\gcd(p,q)\in\{1,2\}$ and any  integer $k$, we have
$$B_{p,q}\cong B_{p,kp\pm q}.$$
\end{lemma}
\begin{proof}
First note that by reversing the first diagram in either \Cref{fig:Bpq} or \ref{fig:Apq}, we see that $B_{p,q}\cong B_{p,p-q}$.

Now take $k\in\Z$.  Start with $B_{p,kp+q}$ as in the second diagram of either \Cref{fig:Bpq} or \ref{fig:Apq}, with the dotted circle replaced by a $\langle0\rangle$-framed circle as in the proof of Lemma $\ref{lem:Bpq}$.  Then the  diffeomorphism $B_{p,q}\cong B_{p,kp+ q}$ follows by performing Rolfsen twists \cite[Section 5.3]{GS} along this $\langle0\rangle$-framed circle.
%
%
\end{proof}


\section{The 2-Farey tree}\label{s:2Fareytree}

We define a variant of the classical Farey tree \cite{aigner}, which we refer to as the 2-Farey tree.  This is a binary tree with ordered triples of fractions, called 2-Farey triples, at each node.  The triple at the root is $\left(\frac10,\frac32,\frac22\right)$ and the recursive rule is given by 

{\centering
\includegraphics[width=7cm]{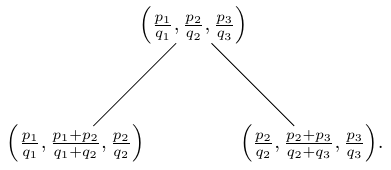}
\par}
\noindent

The first three rows of the tree, and two further nodes, are shown below.

{\centering
\includegraphics[width=10cm]{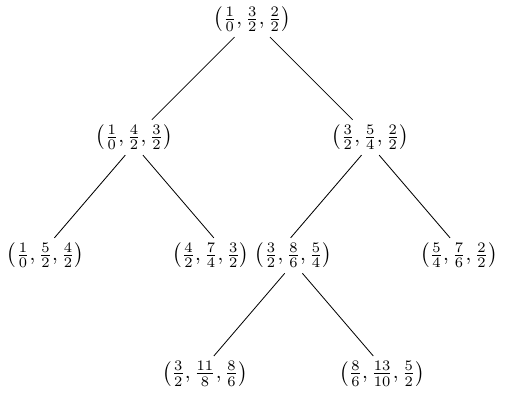}
\par}
\noindent
Observe that each triple contains two reduced fractions and one fraction $p/q$ with $\gcd(p,q)=2$.

For comparison, part of the standard Farey tree is shown below.  It uses the same recursive rule as above but the triple at the root is $\left(\frac10,\frac11,\frac01\right)$.

{\centering
\includegraphics[width=10cm]{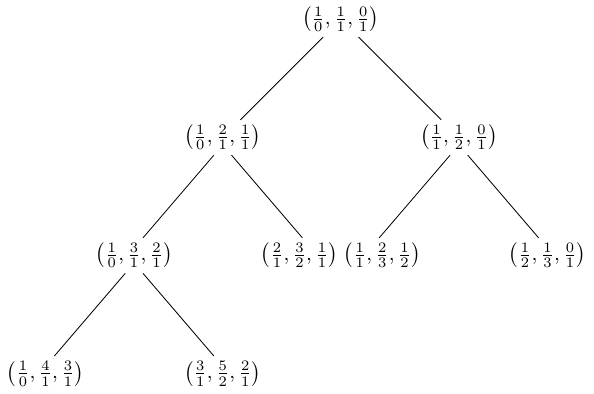}
\par}
\noindent
Note that the 2-Farey tree is obtained from the subtree of the Farey tree on the left, beginning at $\left(\frac10,\frac31,\frac21\right)$, by multiplying all denominators by 2.

The following theorem  follows immediately from the corresponding facts for the classical Farey tree found in the wonderful book by Aigner \cite[Section 3.2]{aigner}.

%

\begin{thm}\label{thm:2Farey}
Every fraction $p/q$ with $q$ even, $p\ge q$, and $\gcd(p,q)\le2$ appears in the $2$-Farey tree.  Every such fraction with $p>2$ is generated as a mediant exactly once.  For each $2$-Farey triple $\left(\dfrac{p_1}{q_1},\dfrac{p_2}{q_2},\dfrac{p_3}{q_3}\right)$, we have  
\begin{equation}\label{eqn:2Farey}
p_2=p_1+p_3,\quad q_2=q_1+q_3,\quad \mathrm{and}\quad p_iq_{i+1}-p_{i+1}q_i=2\ \mathrm{for}\  i\in\{1,2\}.  
\end{equation}
Furthermore:

\begin{enumerate}[(i)]
\item If $p_1>q_1$ and $p_2>q_2$ satisfy $q_i$ even and $p_1q_2-p_2q_1=\pm2$, then $p_1/q_1$ and $p_2/q_2$ are contained in a $2$-Farey triple.
\item For any coprime natural numbers $p_1, p_2$, there exist unique even numbers $q_1$ and $q_2$ satisfying
$0\le q_i\le p_i$ and $p_1q_2-p_2q_1=\pm2$.  After possibly relabeling, we may take $p_1q_2-p_2q_1=2$, and then $\left(\dfrac{p_1}{q_1},\dfrac{p_1+p_2}{q_1+q_2},\dfrac{p_2}{q_2}\right)$ is a $2$-Farey triple.
\end{enumerate}
\end{thm}

\begin{rmk}
One could alternatively choose the triple $\left(\frac20,\frac31,\frac11\right)$ at the root of the 2-Farey tree.  The resulting tree would be obtained from that above by replacing each fraction $p/q$ by $p/(p-q)$, and reflecting the tree and each triple from left to right.  It contains all triples $\left(\dfrac{p_1}{q_1},\dfrac{p_2}{q_2},\dfrac{p_3}{q_3}\right)$ with $p_i\equiv q_i \pmod2$ and  satisfying \eqref{eqn:2Farey}.
Since by \Cref{lem:addp} we have $B_{p,q}=B_{p,p-q}$, this would result in the same embeddings of rational balls in $\CPb$ in the next section.

We could also consider an expanded 2-Farey tree obtained by multiplying denominators by 2 in all nodes of the full Farey tree shown above.  It again follows from \Cref{lem:addp} that this does not yield any more embedding results than those coming from the 2-Farey tree we have selected, though we will see that for the base of our induction, it is helpful to include two more nodes at the base of the tree.
\end{rmk}


\section{Embeddings of 2-Farey triples}\label{s:ANN}

Given an ordered $n$-tuple $\left(\dfrac{p_1}{q_1},\dfrac{p_2}{q_2},\dots,\dfrac{p_n}{q_n}\right)$, with $\gcd(p_i,q_i)\in\{1,2\}$, we define a Kirby diagram of a manifold $X=X_{\frac{p_1}{q_1},\frac{p_2}{q_2},\dots,\frac{p_n}{q_n}}$ as follows.  Let $r_i=1-i/(n+1)$ for $i=1,\dots,n$, so that
$$1>r_1>r_2>\dots>r_n>0,$$
and $\varepsilon=1/(4n)$, so that the intervals $[r_i-\varepsilon,r_i+\varepsilon]$ are pairwise disjoint in $(0,1)$.

The manifold $X$ is a handlebody with one $0$-handle, one $1$-handle, and $n$ $2$-handles.  To draw the diagram, start with a dotted circle drawn as a circle of diameter 3.  For each $1\le i\le n$, draw the framed knot $K_{p_i,q_i}$ from \Cref{s:balls} in an $\varepsilon$-neighbourhood of the torus which is distance $r_i$ from the dotted circle.  An example is shown in Figure \ref{fig:Base}.

\begin{figure}[htbp]
\centering
\includegraphics[width=14cm]{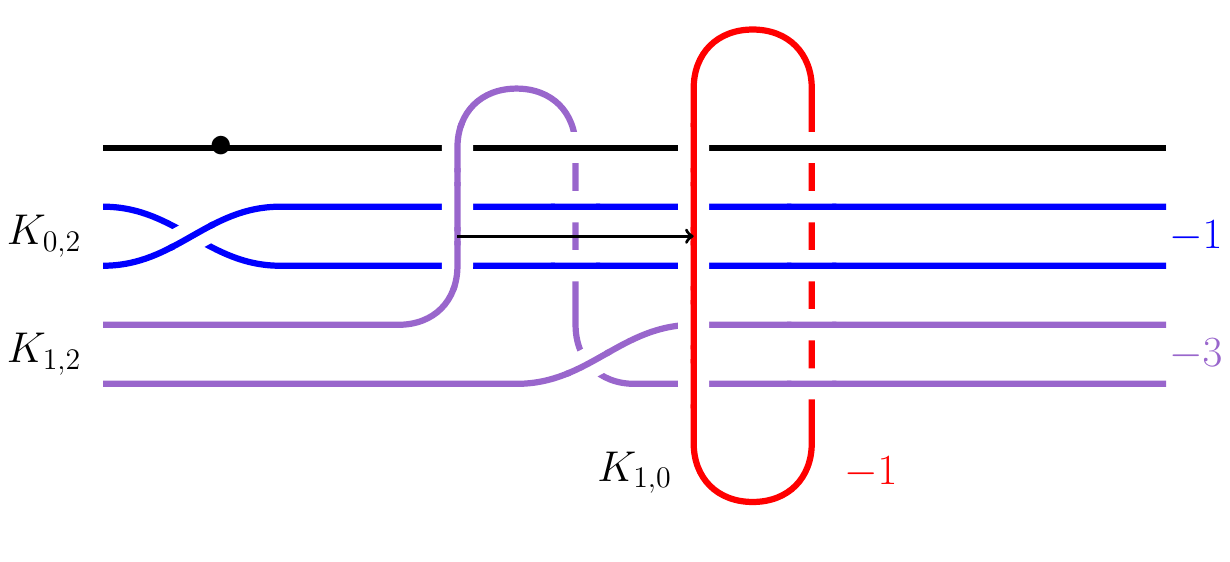}
\caption{{\bf The manifold $X_{\frac10,\frac12,\frac02}$.} The black arrow indicates a handle slide that may be used to simplify the diagram.}
\label{fig:Base}
\end{figure}

\begin{lemma}\label{lem:LPembed}
For any $\left(\dfrac{p_1}{q_1},\dfrac{p_2}{q_2},\dots,\dfrac{p_n}{q_n}\right)$, with $\gcd(p_i,q_i)\in\{1,2\}$, the disjoint union
$$\bigsqcup_{i=1}^n B_{p_i,q_i}$$
embeds smoothly in $X_{\frac{p_1}{q_1},\frac{p_2}{q_2},\dots,\frac{p_n}{q_n}}$.
\end{lemma}
\begin{proof}
This follows exactly as in the proof of \cite[Lemma 4.1]{HD2}.  We give a sketch here for convenience.  Consider the 4-manifold given by a single dotted circle, as found in each of $X_{\frac{p_1}{q_1},\frac{p_2}{q_2},\dots,\frac{p_n}{q_n}}$ and $B_{p_i,q_i}$.  This is a copy of $S^1\times B^3$.  For convenience, view $B^3$ as the unit ball in $\R^3$.  Then a torus of constant distance $r$ from the dotted circle corresponds to $S^1\times S_\zeta$, where $S_\zeta$ is the intersection of the unit sphere with the plane $z=\zeta$.  Here we  take $\zeta(r)$ to be a monotone decreasing bijection from $(0,\infty)$ to $(-1,1)$.

For each $1\le i\le n$, the set of points $(x,y,z)\in B^3$ with $z\in [\zeta(r_i-\varepsilon),\zeta(r_i+\varepsilon)]$ is again a 3-ball, after smoothing corners, which we may denote by $B^3_i$.  The disjoint union $\bigsqcup_{i=1}^n S^1\times B^3_i$ is embedded in $S^1\times B^3$, and since the 2-handle for each $B_{p_i,q_i}$ is attached to the boundary of $S^1\times B^3_i$, the statement follows.
\end{proof}

\begin{prop}\label{prop:2Femb}
For each triple $\left(\dfrac{p_1}{q_1},\dfrac{p_2}{q_2},\dfrac{p_3}{q_3}\right)$ in the $2$-Farey tree, the boundary of $X=X_{\frac{p_1}{q_1},\frac{p_2}{q_2},\frac{p_3}{q_3}}$ is diffeomorphic to $S^1\times S^2$, and adding a $3$-handle and a $4$-handle to $X$ yields $\CPb$.
\end{prop}

\begin{proof} The proof is by structural induction on a slightly enlarged 2-Farey tree, with two added nodes as shown below.

{\centering
\includegraphics[width=6cm]{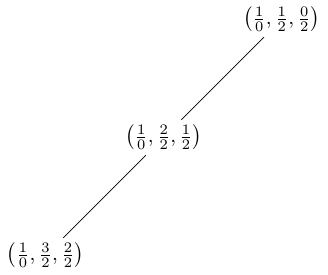}
\par}
\noindent
For the base case, we consider the diagram in \Cref{fig:Base}.  After sliding the purple 2-handle over the red one along the arrow in the figure, we may cancel the red 2-handle and the 1-handle leaving a two-component unlink with framings $-1$ and $0$.  This has boundary $S^1\times S^2$. We add a 3-handle to cancel the 0-framed 2-handle, and then a 4-handle, giving a diagram of $\CPb$.

For the inductive step, we show that if $\left(\dfrac{p_1}{q_1},\dfrac{p_2}{q_2},\dfrac{p_3}{q_3}\right)$ is a 2-Farey triple then
\begin{equation}
\label{eq:slides}
X_{\frac{p_1}{q_1},\frac{p_1+p_2}{q_1+q_2},\frac{p_2}{q_2}}\cong
X_{\frac{p_1}{q_1},\frac{p_2}{q_2},\frac{p_3}{q_3}}\cong
X_{\frac{p_2}{q_2},\frac{p_2+p_3}{q_2+q_3},\frac{p_3}{q_3}}.
\end{equation}
Each of these isomorphisms is established by handlesliding one of the ``outer" 2-handle attaching circles over the middle one.  

\begin{figure}[htbp]
\centering
\includegraphics[width=14cm]{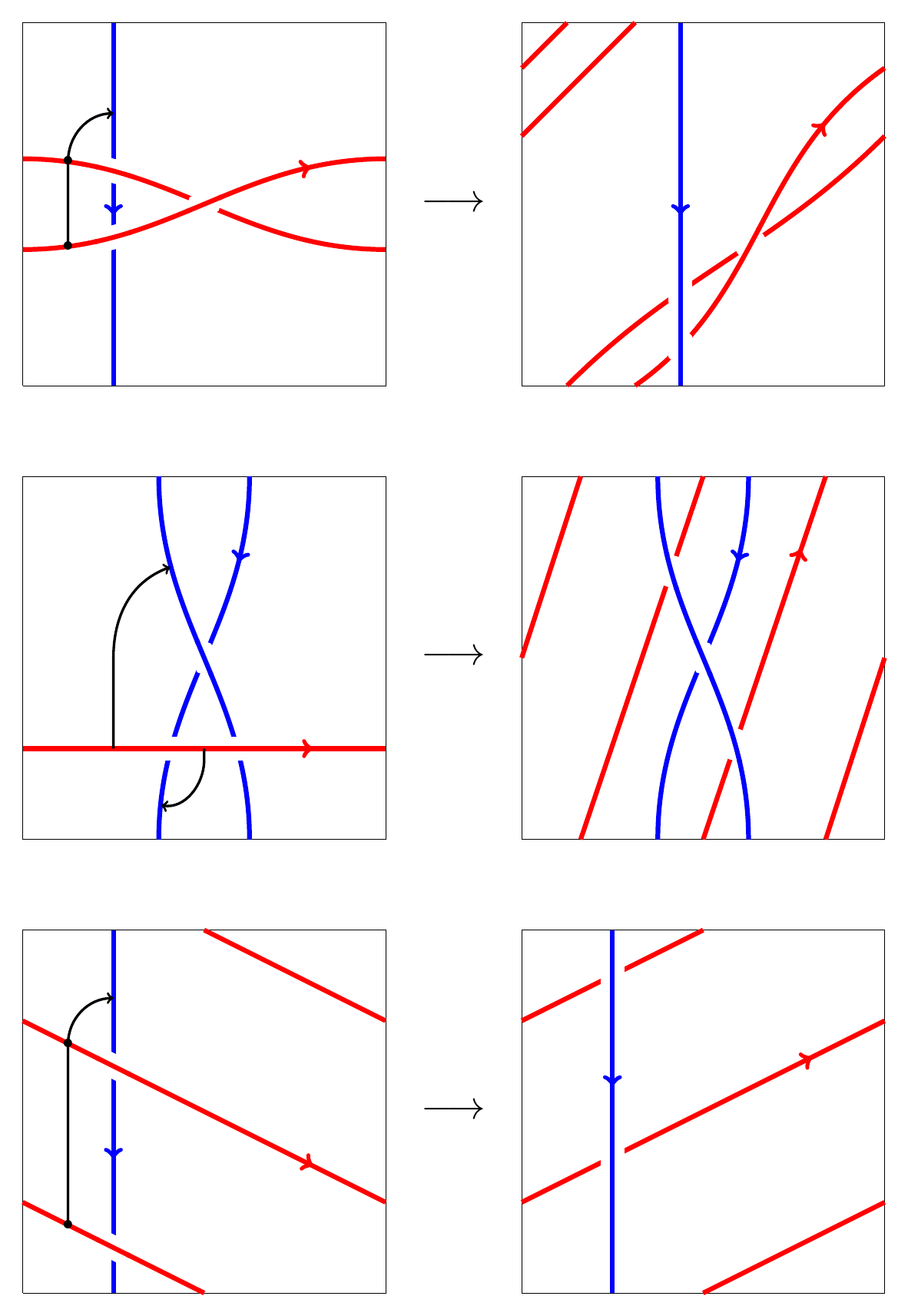}
\caption{{\bf Handlesliding to move down the 2-Farey tree.}  The blue curve is $K_2$.  To move leftward down the tree we take the red curve to be $K_3$, while for the right move, it is $K_1$.  In each case the orientation on the torus is such that $[K_i]\cdot[K_{i+1}]=-2$.} 
\label{fig:slides}
\end{figure}

For convenience, let $K_i\cong K_{p_i,q_i}$ denote the attaching circles of the 2-handles.  The three attaching circles are contained in neighbourhoods of nested tori.  For each successive pair $K_i, K_{i+1}$, the projections to the torus have algebraic intersection $-p_iq_{i+1}+p_{i+1}q_i=-2$.  Either they are both  simple closed curves on the torus, or one is a simple closed curve and the other is a $(2,-1)$ cable, relative to the torus framing, of a simple closed curve.  Thus up to an automorphism of the torus, the pair $K_2, K_3$ appear as one of the diagrams on the left of \Cref{fig:slides}.  Handlesliding as in that figure replaces $K_3$ (the red curve on the left) with $K_3'$ (the red curve on the right).  We see that $K_3'$ lies in an $\varepsilon$-neighbourhood of a torus inside that of $K_2$.  Also $K_3'$ is again either a simple closed curve or a $(2,-1)$ cable, and its homology class on the torus is
$$[K_3']=[K_3]-2[K_2]=(p_3-2p_2,q_3-2q_2)=(-p_1-p_2,-q_1-q_2),$$
noting that $p_2=p_1+p_3$, $q_2=q_1+q_3$, and thus $K_3'$ is the same as $K_{p_1+p_2,q_1+q_2}$.

This establishes the first diffeomorphism in \eqref{eq:slides}.  For the second, we slide $K_1$ over $K_2$.  We again follow \Cref{fig:slides}, but this time with $K_1$ in red.
\end{proof}

\begin{proof}[Proof of \Cref{thm:ANN}]
It follows from \Cref{thm:2Farey} that up to reordering, $\left(\dfrac{p_1}{q_1},\dfrac{p}{q},\dfrac{p_2}{q_2}\right)$ is a 2-Farey triple. The embeddings in $\CPb$ now follow from \Cref{lem:LPembed} and \Cref{prop:2Femb}.
\end{proof}


\begin{proof}[Proof of Corollaries \ref{cor:Bpq} and \ref{cor:Lpq}]
%
%
For this we note that by \Cref{thm:2Farey}, either $\dfrac{p}{q}$ or $\dfrac{p}{p-q}$ appears in a 2-Farey triple.  It follows that $B_{p,q}$ embeds smoothly in $\CPb$ and hence so does its boundary $L(p^2,pq-1)$.  Finally note that an embedding of a 3-manifold into a 4-manifold does not depend on the orientation of either.
\end{proof}


\section{A common framework for known results}\label{s:compare}

In this section we make an observation which attempts to put \Cref{thm:ANN} in a common framework with previous results of Hacking--Prokhorov \cite{hp}, Evans--Smith \cite{es}, and Lisca--Parma \cite{HD2}.  Each of these results gives a binary tree, together with an embedding of a triple of rational balls in $\CP$ for each node of the tree.  The first example of this is the famous Markov tree which contains all triples of solutions of 
\begin{equation*}
p_1^2+p_2^2+p_3^2=3p_1p_2p_3.
\end{equation*}
This features in each of \cite{es,hp,HD2}.  Two further trees of triples arise from \cite{HD2}, and the 2-Farey tree underlies \Cref{thm:ANN}.  The following definition may be interpreted in terms of the proofs of \Cref{prop:2Femb} and  \cite[Proposition 4.6]{HD1}.  In each case we have in mind three curves lying on or near nested tori, and the recursive rule is given by handlesliding one curve over another.

\begin{defn}
A \emph{signed slide triple tree} is an infinite rooted binary tree with a triple $\left(\left(\dfrac{p_1}{q_1},\delta_1\right),\left(\dfrac{p_2}{q_2},\delta_2\right),\left(\dfrac{p_3}{q_3},\delta_3\right)\right)$ at each node and the recursive rule  

{\centering
\includegraphics[width=\textwidth]{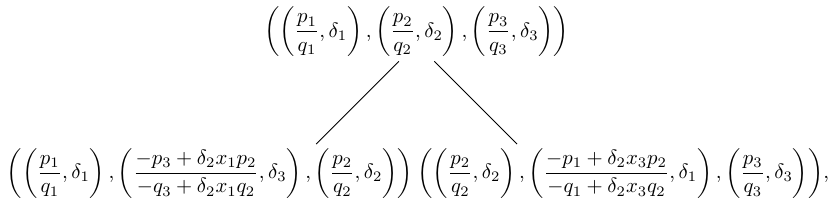}
\par}
\noindent
where $\delta_i\in\{\pm1\}$, $x_1=p_2q_3-p_3q_2$ and $x_3=p_1q_2-p_2q_1$.
\end{defn}

We note that the following gives new recursive descriptions of the triples appearing in \cite[Theorem 1.2(2),(3)]{HD2}.
\begin{thm}\label{thm:trees}
Let $\left(\left(\dfrac{p_1}{q_1},\delta_1\right),\left(\dfrac{p_2}{q_2},\delta_2\right),\left(\dfrac{p_3}{q_3},\delta_3\right)\right)$ be a triple appearing at a node of a signed slide triple tree whose root is labeled by one of
\begin{enumerate}
\item \triple{1}{-1}{1}{5}{1}{1}{2}{1}{1},
\item \triple{1}{1}{-1}{-3}{-1}{-1}{2}{1}{1},
\item \triple{1}{-1}{1}{-3}{-1}{-1}{2}{1}{1},
\item \triple{1}{0}{-1}{3}{2}{-1}{2}{2}{-1}.
\end{enumerate}
Then the disjoint union $\bigsqcup\delta_i B_{p_i,q_i}$ embeds smoothly in $\CP$.  The first root above gives the triples arising from the Markov tree, the second and third are the second and third families of embeddings in \cite[Theorem 1.2]{HD2}, and the fourth gives the triples in \Cref{thm:ANN}.
\end{thm}

\begin{proof}[Proof sketch]
For the case of the first root, which yields the Markov tree, we appeal to \cite{perling}, which refers to $\{q_1,q_2,q_3\}$ as the \emph{T-weights} of the Markov triple $\{p_1,p_2,p_3\}$.  The identities \cite[(1)]{perling} show that $x_1=3p_1$ and $x_3=3p_3$, and then the recursive rule above agrees with the mutation rules for Markov triples and their T-weights \cite[\S2 and Lemma 3.1]{perling}.  The first three rows of the tree are shown below.  We omit the $\delta_i$ which are all equal to $+1$.

{\centering
\includegraphics[width=10cm]{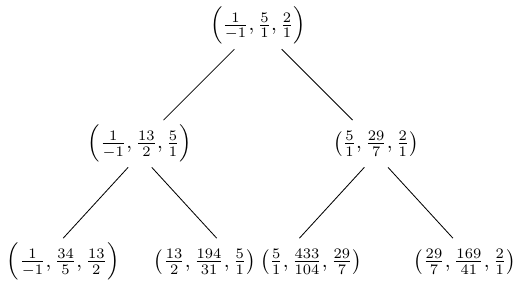}
\par}
\noindent

The second root above yields the triples of rational balls embedding in $\CP$ given by Lisca--Parma in \cite[Theorem 1.2(2)]{HD2}.  
One can show that $x_1=\delta_1p_1$ and $x_3=\delta_3p_3$ and that, as for the standard Markov equation, the recursive rule given above agrees with the mutation rule for the triples of solutions and corresponding weights of the Markov-type equation
\begin{equation*}
\delta_1 p_1^2+\delta_2 p_2^2 +\delta_3 p_3^2=p_1 p_2 p_3,
\end{equation*}
with $\{\delta_1, \delta_2, \delta_3\}=\{1,-1,-1\}$ and weights
\begin{align*}
q_1&\equiv\delta_2p_2/p_3\equiv-\delta_3p_3/p_2\pmod{p_1},\\
q_2&\equiv\delta_3p_3/p_1\equiv-\delta_1p_1/p_3\pmod{p_2},\\
q_3&\equiv\delta_1p_1/p_2\equiv-\delta_2p_2/p_1\pmod{p_3}.
\end{align*}
Each weight $q_i$ is strictly between $0$ and $p_i$ unless $p_i=1$, in which case we take $q_i=1$.

The first entries of the resulting tree of triples are shown below.  In each triple the underlined entry has $\delta_i=1$.  Note that $B_{-p,-q}=B_{p,q}$, but the negative signs in the tree below are needed for the recursion.

{\centering
\includegraphics[width=10cm]{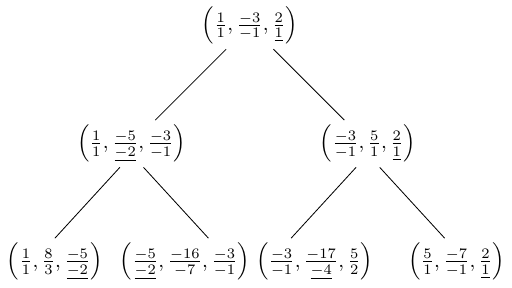}
\par}
\noindent

The first few entries of the tree arising from the third root in the statement are shown below, as well as two extra nodes before the root.  In each triple the underlined entry has $\delta_i=-1$.

{\centering
\includegraphics[width=10cm]{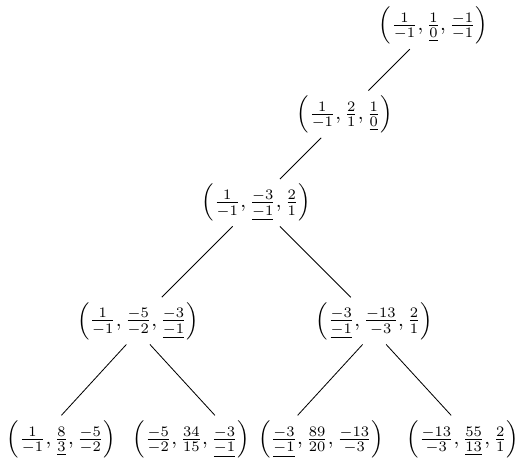}
\par}
\noindent
We claim that this contains precisely the triples of rational balls embedding in $\CP$ given by Lisca--Parma in \cite[Theorem 1.2(3)]{HD2}.  To see this, for each node, let $\gamma_i=(p_i,q_i)$,  and take
$$(x_1,x_2,x_3)=(\gamma_2\cdot\gamma_3,\gamma_1\cdot\gamma_3,\gamma_1\cdot\gamma_2),$$
where $(p,q)\cdot(r,s)=ps-qr$.  For example, if we consider the first triple 
$\left(\frac{1}{-1},\frac{1}{\underline{0}},\frac{-1}{-1}\right)$, we get
$$\gamma_1=(1,-1), \gamma_2=(1,0), \gamma_3=(-1,-1) \text{ and } (x_1,x_2,x_3)=(-1,-2,1),$$
and for the second triple $\left(\frac{1}{-1},\frac{2}{1},\frac{1}{\underline{0}}\right)$ we obtain
$$\gamma_1=(1,-1), \gamma_2=(2,1), \gamma_3=(1,0) \text{ and } (x_1,x_2,x_3)=(-1,-1,-3).$$

We claim that every triple in this tree satisfies the following Markov-type equation
\begin{equation}
\delta_1 x_1^2+\delta_2 x_2^2 +\delta_3 x_3^2=x_1 x_2 x_3-4,\label{eq:3-2}
\end{equation}
and also
\begin{equation}
\delta_1 p_1^2+\delta_2 p_2^2 +\delta_3 p_3^2-\delta_1\delta_2p_1 p_2 x_3-\delta_1\delta_3p_1p_3x_2-\delta_2\delta_3p_2p_3x_1-p_1p_3x_1x_3=0.\label{eq:3-3}
\end{equation}
These are equations (3-2) and (3-3) in \cite{HD2}.

It is straightforward to check these equations are satisfied for the triple $\left(\frac{1}{-1},\frac{1}{\underline{0}},\frac{-1}{-1}\right)$, which forms the base case for a structural induction.  We then check that triples given by the recursive rule continue to satisfy \eqref{eq:3-2} and \eqref{eq:3-3}.

Suppose that we have $\delta_2=-1$ and that
$$\gamma_1=(p_1,q_1), \gamma_2=(p_2,q_2), \gamma_3=(p_3,q_3),$$
and thus
$$(x_1,x_2,x_3)=(p_2q_3-p_3q_2,p_1q_3-p_3q_1,p_1q_2-p_2q_1).$$
By induction, we assume that \eqref{eq:3-2} and \eqref{eq:3-3} hold for these variables.
Moving left down the tree, the recursive rule leads to $\delta_3=-1$,
$$\widehat{\gamma}_1=(p_1,q_1), \widehat{\gamma}_2=(-p_3-x_1p_2,-q_3-x_1q_2), \widehat{\gamma}_3=(p_2,q_2),$$
and 
\begin{equation}
(\widehat{x}_1,\widehat{x}_2,\widehat{x}_3)=(x_1,x_3,-x_2-x_1x_3).\label{eq:mut}
\end{equation}
We verify that \eqref{eq:3-2} holds:
\begin{align*}
&\widehat{x}_1^2+\widehat{x}_2^2-\widehat{x}_3^2-\widehat{x}_1\widehat{x}_2\widehat{x}_3+4\\
&=x_1^2+x_3^2-x_2^2-2x_1x_2x_3-x_1^2x_3^2+x_1x_2x_3+x_1^2x_3^2+4\\
&=x_1^2-x_2^2+x_3^2-x_1x_2x_3+4\\
&=0,
\end{align*}
and also \eqref{eq:3-3}:
\begin{align*}
&\widehat{p}_1^2+\widehat{p}_2^2-\widehat{p}_3^2-\widehat{p}_1 \widehat{p}_2 \widehat{x}_3+\widehat{p}_1\widehat{p}_3\widehat{x}_2+\widehat{p}_2\widehat{p}_3\widehat{x}_1-\widehat{p}_1\widehat{p}_3\widehat{x}_1\widehat{x}_3\\
&=p_1^2+p_3^2+2p_2p_3x_1+p_2^2x_1^2-p_2^2-p_1p_3x_2-p_1p_2x_1x_2-p_1p_3x_1x_3-p_1p_2x_1^2x_3\\
&\quad+p_1p_2x_3-p_2p_3x_1-p_2^2x_1^2+p_1p_2x_1x_2+p_1p_2x_1^2x_3\\
&=p_1^2-p_2^2+p_3^2+p_1p_2x_3-p_1p_3x_2+p_2p_3x_1-p_1p_3x_1x_3\\
&=0,
\end{align*}
in both cases using induction for the final equality.

The other cases (moving right down the tree from $\delta_2=-1$, or starting with $\delta_2=1$) are minor variants of this.

The proof of \cite[Theorem 1.2(3)]{HD2} shows that the triples therein come from solutions to \eqref{eq:3-2} and \eqref{eq:3-3} such that the triples $(x_1,x_2,x_3)$ are obtained from the minimal solution $(1,2,1)$ by mutations, sign changes, and permutations as in \eqref{eq:mut}.  By comparison, we conclude that our  signed slide triple tree contains all of the triples of rational balls in the family given in \cite[Theorem 1.2(3)]{HD2}.

For the last case, one can show by induction that $x_1=x_3=2$ for all nodes, and then the recursion rule is easily seen to agree with that of the 2-Farey tree.
\end{proof}


\section{Further constructions}\label{s:homotopyCP2s}

In this section we describe some methods for finding rational balls bounded by lens spaces which embed in a homotopy $\CP$ or $\CPb$.
These methods are suitable for use in a computer search and led us to the examples in \Cref{thm:ANN}.
They also give rise to other examples which have not yet appeared in the literature.

Each of these homotopy $\CP$'s admits a handle decomposition with at most one 1-handle and at most two 3-handles, which we do not work out explicitly.
As mentioned in the introduction above, we expect that all these homotopy $\CP$'s are in fact standard.

\begin{quest}\label{q:exoticCP2s}
Are all homotopy $\CP$'s constructed in this section diffeomorphic to the standard $\CP$?
\end{quest}

The third construction below, in particular, gives handle decompositions of \emph{geometrically simply-connected} homotopy $\CP$s starting from certain Berge knots.
Furthermore, these handle decompositions have only one 3-handle.

\begin{conj}
The homotopy $\CP$'s constructed in Proposition~\ref{p:berge} below, starting from a Berge knot, are diffeomorphic to $\CP$.
\end{conj}

Since Berge knots are explicitly classified into several families, the conjecture should be provable by analysing each of the families individually.

Throughout the section, we will make use of the following fact, without explicit mention: every lens space which bounds a rational homology ball, also bounds a rational ball constructed with only $0$-, $1$-, and $2$-handles \cite{BBL,Lisca-ribbon}.

\subsection{Cobordisms between lens spaces}

Here we describe two (very similar) methods for constructing a homotopy $\CP$ containing three rational homology balls whose boundaries are lens spaces.

\begin{prop}\label{prop:Addc}
Suppose that we have two rational numbers, $p/q = [a_1,\dots,a_m]$ and $r/s = [b_1,\dots,b_n]$, and an integer $c$ such that:
\begin{itemize}
\item $L(p,q)$ and $L(r,s)$ bound rational homology balls,
\item letting $t/u = [a_m,\dots,a_1,c,b_1,\dots,b_n]$, $L(t,u)$ also bounds a rational homology ball, and
\item $p$, $r$, and $t$ are pairwise coprime.
\end{itemize}
Then there exists a homotopy $\CP$ or $\CPb$, $X$, and an embedding $B \sqcup B' \sqcup (-B'') \hookrightarrow X$, where $B$, $B'$, and $B''$ are three rational homology balls bounding $L(p,q)$, $L(r,s)$, and $L(t,u)$, respectively, constructed  using handles of index at most $2$.

Moreover, $X$ is positive definite, and thus a homotopy $\CP$, if and only if $c-q/p-s/r > 0$.
\end{prop}

Note that in the statement we do 
not impose that $c\ge 2$.

\begin{example}
Using Lisca's classification \cite{Lisca-ribbon} (see also \cite[Appendix A]{AGLL}), each of $L(16,7)$, $L(25,14)$, and $L(1681,737)$ bound rational balls.  In fact using \Cref{lem:Bpq} we have 
\begin{align*}
L(16,7)&=\partial B_{4,2},\\ 
L(25,9)&=\partial B_{5,2}, \quad\mbox{and}\\ 
L(1681,737)&=\partial B_{41,18}.
\end{align*}
We also find that 
\begin{align*}
16/7&=[3,2,2,3],\\
25/9&=[3,5,2], \quad\mbox{and}\\ 
1681/737&=[3,2,2,3,5,3,5,2].
\end{align*}
We conclude from \Cref{prop:Addc} that
$$B_{4,2}\sqcup B_{5,2}\sqcup -B_{41,18}\hookrightarrow X,$$
where $X$ is a homotopy $\CP$.  
\end{example}

\begin{proof}
Call $L = L(p,q)$, $L' = L(r,s)$, and $L'' = L(t,u)$.

In fact, we will prove something stronger, namely that there exists a simply-connected cobordism $W$ from $L\#L'$ to $L''$ with $b_2(W) = 1$.
We will see at the end of the proof below how this implies the statement.

This cobordism is shown in \Cref{fig:Addc}, and it is obtained by attaching a single 2-handle to $[0,1] \times (L\#L')$ along the connected sum $K\# K' \subset L\#L'$, where $K$ and $K'$ are cores of solid tori of a genus-1 Heegaard decomposition of $L$ and $L'$, respectively.

\begin{figure}[htbp]
\centering
\includegraphics[width=14cm]{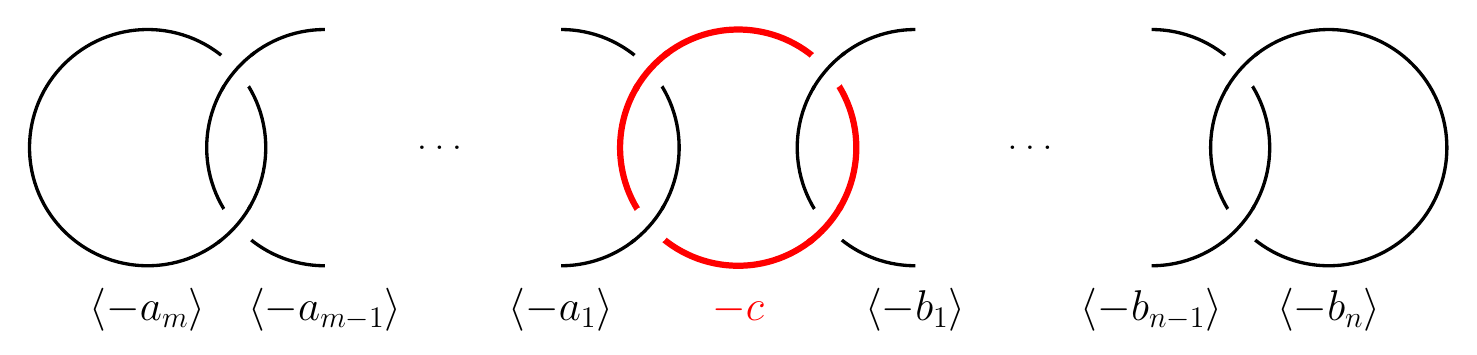}
\caption{{\bf A cobordism from $L\# L'$ to $L''$.} } 
\label{fig:Addc}
\end{figure}

To see that $W$ is simply-connected, we will show that $\pi_1(W)$ is cyclic and $H_1(W) = 0$.

To prove the former, turn $W$ upside-down, and view it as a 2-handle attachment on $[0,1]\times L''$.
Now, $\pi_1(W)$ is a quotient of $\pi_1(L'')$ by Seifert--van Kampen, and the latter is cyclic since $L''$ is a lens space.

To prove that $H_1(W) = 0$, we prove that it is a quotient of $H_1(L) \oplus H_1(L')$. 
We already know that $H_1(W)$ is a quotient of $H_1(L'')$, and since $H_1(L'')$ and $H_1(L) \oplus H_1(L')$ have coprime orders by the third assumption in the statement, then $H_1(W) = 0$.

Since $W$ is obtained by attaching a 2-handle to $[0,1]\times (L\#L')$, $H_*(W,L\#L')$ is supported in degree $2$ (and it is isomorphic to $\Z$, but we will not need this).
By the long exact sequence of the pair, the inclusion induces a surjection $H_1(L\#L') \to H_1(W)$.

The claim that $b_2(W) = 1$ also quickly follows from the long exact sequence of the pair.


Now view $W$ as a cobordism from the empty set to $-(L\# L') \sqcup L''$.
First attach a 3-handle along the connected-sum sphere in $-(L\# L')$, and then attach $B$ along $-L$, $B'$ along $-L'$, and $-B''$ along $L''$.
Since $B$, $B'$ and $B''$ contain no $3$- or $4$-handles, and we have turned the handlebodies upside-down so that they are built by attaching handles to their boundary,
it follows that we have only attached $2$-, $3$-, and $4$-handles to $W$, and in particular the resulting closed 4-manifold, $X$, is simply-connected.

One quickly verifies that $X$ is a homology $\CP$ or $\CPb$, by additivity of the Euler characteristic, and since $X$ is simply-connected it also a homotopy $\CP$ or $\CPb$. 

We now prove the final statement. After performing bracketed blow-ups and blow-downs in Figure~\ref{fig:Addc}, we can suppose that $a_i \ge 2$ and $b_j \ge 2$ for each $i$ and $j$.

We claim that the quantity $c-q/p-s/r$ is invariant under these operations.
One way of seeing it is by explicitly computing how $c-q/p-r/s$ varies with these blow-ups and blow-downs: blow-ups and blow-downs in the interior of the $a$-chain or of the $b$-chain leave $c$, $q/p$, and $s/r$ unchanged, whereas blow-ups and blow-downs around the $c$-framed unknot change either $q/p$ or $s/r$ by $\pm1$, and $c$ by $\mp1$, thus leaving the quantity unchanged.
A different way of seeing this, which is useful below, is to  notice that $c-q/p-s/r$ is the Euler number of the Seifert fibration on $L(t,u)$ associated to the plumbing in Figure~\ref{fig:Addc}, with central vertex the red component.
This quantity is an invariant of the 3-manifold \emph{with the given Seifert fibration}.
However, blowing up and down the bracketed unknots corresponds exactly to preserving the Seifert fibration, and therefore its Euler number.

The final statement now follows by comparing the signatures of plumbed 4-manifolds obtained by removing the brackets in \Cref{fig:Addc}, and thus determining whether $W$ and hence $X$ is positive or negative definite.
From the discussion above, $W$ is the difference between the plumbing associated to a Seifert fibration (with two singular fibres and $e_0 = c$) and the two plumbings associated to the two legs.
Since the latter are negative definite (as they are the canonical plumbings of $L(p,q)$ and $L(r,s)$), $W$ is negative definite if and only if the Euler number $c - q/p - s/r$ of the fibration is negative, which concludes the proof.
\end{proof}

\begin{prop}\label{prop:Add4}
Suppose that we have two rational numbers, $p/q = [a_1,\dots,a_m]$ and $t/u = [b_1,\dots,b_m]$ and an index $1 \le j \le m$ such that
\begin{itemize}
\item $L(p,q)$ and $L(t,u)$ bound rational homology balls,
\item $p$ and $t$ are odd and pairwise coprime, and
\item for $i \neq j$, $a_i = b_i$, and $b_j = a_j + 4$.
\end{itemize}
Then there exists a homotopy $\CP$, $X$, and an embedding $B \sqcup B_{2,1} \sqcup (-B'') \hookrightarrow X$, where $B$ and $B''$ are two rational homology balls bounding $L(p,q)$ and $L(t,u)$, respectively, constructed  using handles of index at most $2$.
\end{prop}

As in the previous proposition, in this statement we do \emph{not} require that $a_j\ge 2$ or $b_j \ge 4$.

\begin{example}
We have $9/4=[3,2,2,2]$, $49/20=[3,2,6,2]$,  $L(9,4)=-\partial B_{3,1}$, and $L(49,20)=\partial B_{7,3}$.  We conclude from \Cref{prop:Add4} that
\[
B_{2,1}\sqcup -B_{3,1}\sqcup -B_{7,3}\hookrightarrow X,
\]
where $X$ is a homotopy $\CP$. 
\end{example}

Since the proof is very similar to the previous one, we only give a very quick sketch.

\begin{proof}[Proof sketch]
Call $L = L(p,q)$, $L' = L(4,1)$, and $L'' = L(t,u)$, and let $B$, $B'$, and $B''$ be rational homology balls with boundaries $L$, $L'$, and $L''$, respectively, each constructed using  handles of index at most 2.

We construct a cobordism $W$ from $L \# L'$ to $L''$ which has $\pi_1(W) = 1$ and $b_2(W) = 1$.
This cobordism is shown in \Cref{fig:Add4}.

\begin{figure}[htbp]
\centering
\includegraphics[width=14cm]{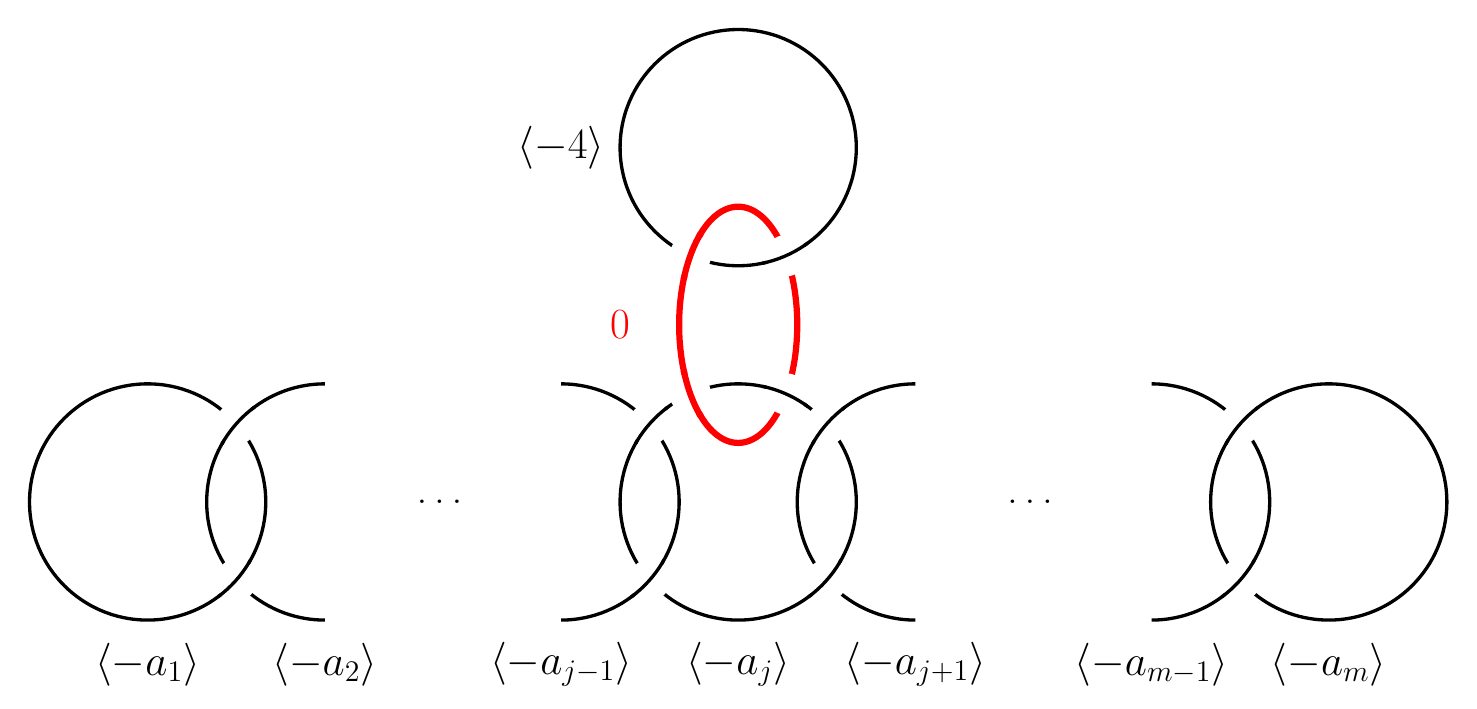}
\caption{{\bf A cobordism from $L\# L(4,1)$ to $L''$.} } 
\label{fig:Add4}
\end{figure}

One proves the two properties in exactly the same way as for \Cref{prop:Addc}: $\pi_1(W)$ is cyclic since it is a quotient of $\pi_1(L'')$, and it is trivial because of the condition on the coprimality of the orders of $H_1(L)\oplus H_1(L')$ and $H_1(L'')$, and the long exact sequence of the pair.
Moreover, $b_2(W) = 1$ because $W$ is a 2-handle cobordism between rational homology spheres.

Gluing in $B$, $B'$, and $-B''$ preserves these two properties and yields the desired 4-manifold $X$ which is a homotopy $\CP$ or $\CPb$, with an embedding of $B \sqcup B' \sqcup -B''$.

To pin down the sign, note that no rational homology ball $B'$ with boundary $L(4,1)$ can embed in a homotopy $\CPb$ (e.g. as a corollary of Donaldson's diagonalisation theorem~\cite{JPP, GOobst}), so $X$ is a homotopy $\CP$.
\end{proof}

\subsection{Berge knots}

Knot traces give us another source of embeddings of lens spaces in homotopy $\CP$s.

\begin{prop}\label{p:berge}
Let $K \subset S^3$ be a knot, $p = m^2$ a square, and $Y = S^3_{p}(K)$.
\begin{itemize}
\item[(i)] If $Y$ is a lens space which bounds a rational homology ball, then there exist a rational homology ball $B$ bounding $-Y$ and a homotopy $\CP$, $X$, such that $B$ embeds in $X$.
\item[(ii)] If $Y = L' \# L''$ is a connected sum of two lens spaces $L'$ and $L''$, each bounding a rational homology ball, then there exist rational homology balls $B'$ and $B''$ bounding $-L'$ and $-L''$, respectively, and a homotopy $\CP$, $X$, such that $B'\sqcup B''$ embeds in $X$.
\end{itemize}
In fact, we can choose any balls $B$, $B'$, or $B''$ that are constructed  using handles of index at most $2$.
\end{prop}

This was essentially proven in~\cite{AGLL}, along the same line as the proofs in the previous section.
Below we sketch a proof for completeness.

\begin{example}
Let $r\ge 2$ be an integer.  The first two examples below are special cases of \Cref{thm:ANN}, but the third one is not.
\begin{itemize}
\item[(i)] $S^3_{(2r+2)^2}(T(2r+1,2r+3)) = L((2r+2)^2, 4r+3) = \de B_{2r+2,2}$.
This implies that there is an embedding of $-B_{2r+2,2}$ in a homotopy $\CP$ for each $r$.
\item[(ii)] $S^3_{r^2(r+1)^2}T(r^2,(r+1)^2) = L(r^2,r^2-2r-1) \# L((r+1)^2,2r+1)$.
These two lens spaces do indeed bound rational homology balls: $L(r^2,r^2-2r-1) = \de B_{r,r-2}$ and $L((r+1)^2,2r+1) = \de B_{r+1,2}$.
By the proposition above, $-B_{r,r-2}\sqcup {-B_{r+1,2}}$ disjointly embed in a homotopy $\CP$.
\item[(iii)] To give an example that does not come from a torus knot, we can look at $L(49,31)$: this is a surgery along a Berge knot (in family VII)~\cite{Berge,Jake}, but it does not arise as a surgery along a torus knot\footnote{Since $49 = 48+1 = 50-1$, the only lens spaces $L(49,q)$ that arise as integer surgeries along torus knots are $S^3_{49}(T(3,16)) = L(49,40)$ and $S^3_{49}(T(2,25)) = L(49,4)$~\cite{Moser}.}.
Moreover, it bounds a rational homology ball by~\cite{Lisca-ribbon}.
Therefore, there is a rational homology ball $B$ whose boundary is $L(49,18)$ (for instance, one of those constructed by Lisca~\cite{Lisca-ribbon}) that embeds in a homotopy $\CP$.
\end{itemize}
\end{example}

\begin{proof}[Proof of Proposition~\ref{p:berge}]
We prove (i) and (ii) simultaneously.
In case (i), choose $B$ to be any rational homology ball whose boundary is $-Y$, which is built only using handles of indices at most $2$. (Such balls exist by Lisca's work.)
In case (ii), choose $B'$ and $B''$ constructed with only handles of indices at most $2$, and let $B$ denote their boundary connected sum.
Note that $B$ contains $B' \sqcup B''$.
Denote with $T$ the trace of $p$-surgery along $K$, whose boundary is $-\de B$, by definition.

We claim that $X:=T \cup B$ is a homotopy $\CP$.
First, $X$ is simply-connected, since it has a handle decomposition without 1-handles: we glue $B$ by turning it upside down, so that it has a handle decomposition, relative to its boundary, with only handles of indices \emph{at least} 2.
On the other hand, $T$ is built only with a 0- and a 2-handle, so it is geometrically simply connected.
It follows that $X$ is simply connected.

By additivity of the Euler characteristic, $\chi(X) = 3$, and since it is closed, oriented, and simply-connected, $b_2(X) = 1$.
Finally, since $X$ contains $T$, which in turn contains a surface of self-intersection $p > 0$ (say, obtained by capping off a Seifert surface of $K$ with the core of the 2-handle), then $X$ is positive definite.
That is, $X$ is a homotopy $\CP$.
\end{proof}

\section{Data}\label{s:data}

\Cref{table:triples} contains \emph{all} triples of pairs of integers, $(p_1,q_1)$, $(p_2,q_2)$, $(p_3,q_3)$ such that:
\begin{itemize}
\item there exist rational homology balls $B_1$, $B_2$, and $B_3$ with boundaries $L(p_1,q_1)$, $L(p_2,q_2)$, and $L(p_3,q_3)$, respectively, that are unobstructed from disjointly embedding in $\CP$ according to~\cite{GOobst} (see also~\cite{JPP}), and
\item $4\le p_1 < p_2 < p_3 \le 256$.
\end{itemize}

\begin{table}[ht]
\begin{center}
\small
\begin{tabular}[t]{cccl}
\toprule
$p_1,q_1$ & $p_2,q_2$ & $p_3,q_3$ & Realised \\
\midrule
4,1 & 9,4 & 25,6 & yes$^{*,\dagger}$ \\
4,1 & 9,4 & 49,22 & yes \\
4,1 & 9,4 & 121,34 & yes \\
4,1 & 9,4 & 169,38 & yes \\
4,1 & 9,4 & 169,66 & yes \\
4,1 & 25,6 & 49,8 & yes$^{*,\dagger}$ \\
4,1 & 25,6 & 169,40 & yes \\
4,1 & 25,11 & 81,37 & yes \\
4,1 & 25,11 & 121,45	 & yes \\
4,1 & 25,11 & 169,61 & yes \\
4,1 & 49,8 & 81,10 & yes$^{*,\dagger}$ \\
4,1 & 49,15 & 169,53 & yes \\
4,1 & 49,15 & 225,61 &	yes \\
4,1 & 49,22 & 121,56 & yes \\
4,1 & 81,10 & 121,12 & yes$^{*,\dagger}$ \\
4,1 & 81,37 & 169,79 & yes \\
4,1 & 121,12 & 169,14 & yes$^{*,\dagger}$ \\
4,1 & 121,56 & 225,106	 & yes \\
4,1 & 169,14 & 225,16 & yes$^{*,\dagger}$ \\
9,4 & 16,9 & 49,22 & yes$^*$ \\
9,4 & 25,6 & 64,17 & yes$^*$ \\
9,4 & 25,6 & 64,25 & no \\
9,4 & 25,6 & 256,113 & no\\
9,4 & 25,9 & 256,113 & yes$^{\dagger}$\\
9,4 & 25,11 & 49,15 & no \\
9,4 & 25,11 & 49,18 & no \\
9,4 & 25,11 & 49,22 & no \\
9,4 & 25,11 & 121,45	 & no \\
9,4 & 25,11 & 121,71	 & no \\
9,4 & 25,11 & 169,66 &	no \\
9,4 & 25,11 & 256,113 & no\\
9,4 & 49,15 & 64,25 & no \\
9,4 & 49,15 & 169,53 & no \\
\end{tabular}%
\qquad
\begin{tabular}[t]{cccl}
\toprule
$p_1,q_1$ & $p_2,q_2$ & $p_3,q_3$ & Realised \\
\midrule
9,4 & 49,18 & 64,25 & yes \\
9,4 & 49,22 & 64,25 & no \\
9,4 & 49,22 & 100,41 &yes$^*$ \\
9,4 & 64,17 & 121,34 & yes$^*$ \\
9,4 & 64,25 & 121,34 & no \\
9,4 & 64,25 & 121,71 & no \\
9,4 & 100,41 & 169,66 & yes$^*$ \\
9,4 & 121,34 & 196,57 & yes$^*$ \\
9,4 & 169,66 & 256,97 & yes$^*$ \\
16,9 & 25,11 & 81,37 & yes$^*$ \\
16,9 & 49,22 & 121,56 & yes$^*$ \\
16,9 & 81,37 & 169,79 & yes$^*$ \\
16,9 & 121,56 & 225,106 &yes$^*$ \\
25,6 & 49,8 & 144,25 & yes$^*$ \\
25,6 & 64,17 & 169,40 & yes$^*$ \\
25,6 & 64,25 & 169,40 & no \\
25,11 & 36,13 & 121,45 & 	yes$^*$ \\
25,11 & 49,15 & 64,25 & no \\
25,11 & 49,15 & 81,37 & no \\
25,11 & 49,18 & 64,25 & yes \\
25,11 & 49,22 & 64,25 & no \\
25,11 & 81,31 & 121,71 & 	no \\
25,11 & 81,37 & 196,85 &	yes$^*$ \\
25,11 & 121,45 & 256,97 & yes$^*$\\
36,13 & 49,15 & 169,53 & yes$^*$ \\
49,8 & 81,10 & 256,33 & yes$^*$\\
49,15 & 64,17 & 225,61 & yes$^*$ \\
49,15 & 64,25 & 225,61	 & no \\
49,22 & 81,19 & 121,56 & no \\
64,25 & 81,31 & 121,71 &	yes \\
81,37 & 121,23 & 169,79 &	no \\
121,56 & 169,27 & 225,106 & no \\
\end{tabular}%
\end{center}
\medskip
\caption{{\bf Triples of lens spaces which may embed disjointly in $\CP$.}}
\label{table:triples}
\end{table}

The last column contains a `yes' if there is such an embedding in a homotopy $\CP$, and `no' if we could not find any such embedding with the help of Proposition~\ref{prop:Addc} or~\ref{prop:Add4}. 
An asterisk on `yes' indicates that this arises from a 2-Farey triple as in \Cref{thm:ANN}, while a dagger indicates that this arises in Lisca and Parma's construction with horizontal decompositions \cite{HD2}.

Note that Lisca and Parma construct embeddings of triples of rational balls of type $\pm B_{p,q}$ with $\gcd(p,q) = 1$, so the triples produced in Theorem~\ref{thm:ANN} and the ones from~\cite{HD2} can only overlap if one of the rational balls is $\pm B_{2,0} = \mp B_{2,1}$, with boundary $\mp L(4,1)$.

\bibliographystyle{amsplain}
\bibliography{packings}

\begin{bibdiv}
\begin{biblist}

\bib{AGLL}{unpublished}{
      author={Aceto, Paolo},
      author={Golla, Marco},
      author={Larson, Kyle},
      author={Lecuona, Ana~G.},
       title={Surgeries on torus knots, rational balls, and cabling},
        date={2020},
        note={arXiv:2008.06760},
}

\bib{aigner}{book}{
      author={Aigner, Martin},
       title={Markov's theorem and 100 years of the uniqueness conjecture},
   publisher={Springer, Cham},
        date={2013},
        ISBN={978-3-319-00887-5; 978-3-319-00888-2},
         url={https://doi.org/10.1007/978-3-319-00888-2},
        note={A mathematical journey from irrational numbers to perfect
  matchings},
}

\bib{BBL}{article}{
      author={Baker, Kenneth~L.},
      author={Buck, Dorothy},
      author={Lecuona, Ana~G.},
       title={Some knots in {$S^1\times S^2$} with lens space surgeries},
        date={2016},
        ISSN={1019-8385},
     journal={Comm. Anal. Geom.},
      volume={24},
      number={3},
       pages={431\ndash 470},
  url={https://doi-org.ezproxy2.lib.gla.ac.uk/10.4310/CAG.2016.v24.n3.a1},
}

\bib{Berge}{unpublished}{
      author={Berge, John},
       title={Some knots with surgeries yielding lens spaces},
        date={c. 1990},
        note={arXiv:1802.09722},
}

\bib{EMPR}{unpublished}{
      author={Etnyre, John~B.},
      author={Min, Hyunki},
      author={Piccirillo, Lisa},
      author={Roy, Agniva},
       title={Small symplectic caps and embeddings of homology balls in the
  complex projective plane},
        date={2023},
        note={arXiv:2305.16207},
}

\bib{ER}{article}{
      author={Etnyre, John~B.},
      author={Roy, Agniva},
       title={Symplectic fillings and cobordisms of lens spaces},
        date={2021},
        ISSN={0002-9947},
     journal={Trans. Amer. Math. Soc.},
      volume={374},
      number={12},
       pages={8813\ndash 8867},
         url={https://doi-org.ezproxy1.lib.gla.ac.uk/10.1090/tran/8474},
}

\bib{es}{article}{
      author={Evans, Jonathan~David},
      author={Smith, Ivan},
       title={Markov numbers and {L}agrangian cell complexes in the complex
  projective plane},
        date={2018},
     journal={Geom. Topol.},
      volume={22},
      number={2},
       pages={1143\ndash 1180},
         url={https://doi.org/10.2140/gt.2018.22.1143},
}

\bib{GOobst}{unpublished}{
      author={Golla, Marco},
      author={Owens, Brendan},
       title={On a class of 4-manifolds, after {K}oll{\'a}r},
        date={2025},
}

\bib{GS}{book}{
      author={Gompf, Robert~E.},
      author={Stipsicz, Andr\'{a}s~I.},
       title={{$4$}-manifolds and {K}irby calculus},
      series={Graduate Studies in Mathematics},
   publisher={American Mathematical Society, Providence, RI},
        date={1999},
      volume={20},
        ISBN={0-8218-0994-6},
         url={https://doi-org.ezproxy2.lib.gla.ac.uk/10.1090/gsm/020},
}

\bib{hp}{article}{
      author={Hacking, Paul},
      author={Prokhorov, Yuri},
       title={Smoothable del {P}ezzo surfaces with quotient singularities},
        date={2010},
        ISSN={0010-437X},
     journal={Compos. Math.},
      volume={146},
      number={1},
       pages={169\ndash 192},
         url={http://dx.doi.org/10.1112/S0010437X09004370},
}

\bib{Hantzsche}{article}{
      author={Hantzsche, W.},
       title={Einlagerung von {M}annigfaltigkeiten in euklidische {R}\"{a}ume},
        date={1938},
        ISSN={0025-5874},
     journal={Math. Z.},
      volume={43},
      number={1},
       pages={38\ndash 58},
         url={https://doi-org.ezproxy1.lib.gla.ac.uk/10.1007/BF01181085},
}

\bib{JPP}{article}{
      author={Jo, Woohyeok},
      author={Park, Jongil},
      author={Park, Kyungbae},
       title={Algebraic {Montgomery}-{Yang} problem and smooth obstructions},
        date={2025},
     journal={Trans. Am. Math. Soc.},
      volume={378},
      number={4},
       pages={2969\ndash 3003},
}

\bib{Kollar}{article}{
      author={Koll\'{a}r, J\'{a}nos},
       title={Is there a topological {B}ogomolov-{M}iyaoka-{Y}au inequality?},
        date={2008},
        ISSN={1558-8599},
     journal={Pure Appl. Math. Q.},
      volume={4},
      number={2, Special Issue: In honor of Fedor Bogomolov. Part 1},
       pages={203\ndash 236},
  url={https://doi-org.ezproxy1.lib.gla.ac.uk/10.4310/PAMQ.2008.v4.n2.a1},
}

\bib{lm}{article}{
      author={Lekili, Yank{\i}},
      author={Maydanskiy, Maksim},
       title={The symplectic topology of some rational homology balls},
        date={2014},
        ISSN={0010-2571},
     journal={Comment. Math. Helv.},
      volume={89},
      number={3},
       pages={571\ndash 596},
         url={http://dx.doi.org/10.4171/CMH/327},
}

\bib{Lisca-ribbon}{article}{
      author={Lisca, Paolo},
       title={Lens spaces, rational balls and the ribbon conjecture},
        date={2007},
     journal={Geom. Topol.},
      volume={11},
       pages={429\ndash 472},
}

\bib{LiscaSympFill}{article}{
      author={Lisca, Paolo},
       title={On symplectic fillings of lens spaces},
        date={2008},
        ISSN={0002-9947},
     journal={Trans. Amer. Math. Soc.},
      volume={360},
      number={2},
       pages={765\ndash 799},
}

\bib{LPStein}{article}{
      author={Lisca, Paolo},
      author={Parma, Andrea},
       title={On {S}tein rational balls smoothly but not symplectically
  embedded in {$\Bbb{CP}^2$}},
        date={2022},
        ISSN={0024-6093},
     journal={Bull. Lond. Math. Soc.},
      volume={54},
      number={3},
       pages={949\ndash 960},
         url={https://doi-org.ezproxy1.lib.gla.ac.uk/10.1112/blms.12607},
}

\bib{HD1}{article}{
      author={Lisca, Paolo},
      author={Parma, Andrea},
       title={Horizontal decompositions, {I}},
        date={2024},
        ISSN={1472-2747},
     journal={Algebr. Geom. Topol.},
      volume={24},
      number={6},
       pages={3503\ndash 3525},
         url={https://doi-org.ezproxy1.lib.gla.ac.uk/10.2140/agt.2024.24.3503},
}

\bib{HD2}{article}{
      author={Lisca, Paolo},
      author={Parma, Andrea},
       title={Horizontal decompositions, {II}},
        date={2025},
     journal={Alg. Geom. Topol.},
      volume={25},
       pages={2179\ndash 2207},
        note={arXiv:2302.14606},
}

\bib{Manetti}{article}{
      author={Manetti, Marco},
       title={Normal degenerations of the complex projective plane},
        date={1991},
        ISSN={0075-4102},
     journal={J. Reine Angew. Math.},
      volume={419},
       pages={89\ndash 118},
         url={https://doi-org.ezproxy2.lib.gla.ac.uk/10.1515/crll.1991.419.89},
}

\bib{Moser}{article}{
      author={Moser, Louise},
       title={Elementary surgery along a torus knot},
        date={1971},
        ISSN={0030-8730},
     journal={Pacific J. Math.},
      volume={38},
       pages={737\ndash 745},
  url={http://projecteuclid.org.ezproxy1.lib.gla.ac.uk/euclid.pjm/1102969920},
}

\bib{balls}{article}{
      author={Owens, Brendan},
       title={Equivariant embeddings of rational homology balls},
        date={2018},
        ISSN={0033-5606,1464-3847},
     journal={Q. J. Math.},
      volume={69},
      number={3},
       pages={1101\ndash 1121},
         url={https://doi.org/10.1093/qmath/hay016},
}

\bib{nonsymp}{article}{
      author={Owens, Brendan},
       title={Smooth, nonsymplectic embeddings of rational balls in the complex
  projective plane},
        date={2020},
        ISSN={0033-5606},
     journal={Q. J. Math.},
      volume={71},
      number={3},
       pages={997\ndash 1007},
         url={https://doi-org.ezproxy1.lib.gla.ac.uk/10.1093/qmathj/haaa013},
}

\bib{perling}{unpublished}{
      author={Perling, Markus},
       title={Unfocused notes on the {M}arkoff equation and {T}-singularities},
        date={2022},
        note={arXiv:2210.12982},
}

\bib{PPP}{incollection}{
      author={Popescu-Pampu, Patrick},
       title={The geometry of continued fractions and the topology of surface
  singularities},
        date={2007},
   booktitle={Singularities in geometry and topology 2004},
      series={Adv. Stud. Pure Math.},
      volume={46},
   publisher={Math. Soc. Japan, Tokyo},
       pages={119\ndash 195},
         url={https://doi.org/10.2969/aspm/04610119},
}

\bib{Jake}{unpublished}{
      author={Rasmussen, Jacob},
       title={Lens space surgeries and l-space homology spheres},
        date={2007},
        note={arXiv:0710.2531},
}

\bib{wahl}{article}{
      author={Wahl, Jonathan~M.},
       title={Elliptic deformations of minimally elliptic singularities},
        date={1980},
        ISSN={0025-5831},
     journal={Math. Ann.},
      volume={253},
      number={3},
       pages={241\ndash 262},
         url={http://dx.doi.org/10.1007/BF0322000},
}

\end{biblist}
\end{bibdiv}

\end{document}